\documentclass[a4paper, 11pt]{article}
\usepackage[utf8]{inputenc}
\usepackage[T1]{fontenc}
\usepackage[english]{babel}
\usepackage{fancyhdr}
\usepackage{lipsum}

\usepackage{pgf,tikz-cd,pgfplots}
\pgfplotsset{compat=newest}
\usetikzlibrary{arrows}
\usetikzlibrary{tikzmark}
\usepackage{multicol}
\usepackage{placeins}
\usepackage{delimset}
\usepackage[useregional]{datetime2} 
\usepackage{enumitem} 
\usepackage{lmodern}
\usepackage{amsmath}
\usepackage{amsfonts}
\usepackage{amssymb} 
\usepackage{leftindex}
\usepackage{graphicx} 
\usepackage{threeparttable}
\usepackage{float} 
\usepackage{caption}
\usepackage{geometry} 
\usepackage{tcolorbox}
\usepackage{empheq}
\usepackage{mathrsfs}
\usepackage{newenviron} 
\usepackage{amsthm} 
\usepackage{listings}
\usepackage{hyperref} 
\usepackage{upgreek} 
\pagecolor{white}
\makeindex 
\usepackage{appendix}
\usepackage{booktabs}

\renewcommand{\leq}{\leqslant}
\renewcommand{\geq}{\geqslant}

\newcommand{\R}{\mathbb{R}}
\newcommand{\N}{\mathbb{N}}
\newcommand{\Q}{\mathbb{Q}}

\newcommand{\C}{\mathbb{C}}

\newcommand{\normm}[1]{\left\lVert#1\right\rVert} 

\newcommand{\scalarprod}[1]{\left\langle#1\right\rangle} 
\newcommand{\supp}[1]{\text{supp}\left(#1\right)}

\title{Recovery of a Measure-valued Source in the Heat Equation from Sparse Boundary Measurements}
\author{ Ulysse Dalmasso\thanks{
    Univ Rouen Normandie, CNRS, Normandie Univ, LMRS UMR 6085, F-76000 Rouen, France.\\
    Email: yavar.kian@univ-rouen.fr, ulysse.dalmasso@univ-rouen.fr,  siyu.cen@univ-rouen.fr
}\hspace*{1cm}
    Siyu Cen\footnotemark[1]\hspace*{1cm} Yavar Kian\footnotemark[1]
}
\date{}

\begin{document}

\DTMsetdatestyle{ddmmyyyy} 
\DTMsetup{datesep=/} 

\let\oldproofname=\proofname
\renewcommand{\proofname}{\rm\bf{\oldproofname}}

\theoremstyle{plain}
\newtheorem{thm}{Theorem}[section]
\newtheorem{prop}[thm]{Proposition}
\newtheorem{lem}[thm]{Lemma}
\newtheorem{cor}[thm]{Corollary}
\newtheorem{prpr}[thm]{Property}
\newtheorem{prprs}[thm]{Properties}
\newtheorem*{remark2}{Remark}
\newtheorem{example}{Example}

\theoremstyle{definition}
\newtheorem{df}[thm]{Definition}
\newtheorem*{remark}{Remark}
\newtheorem*{nota}{Notation}
\newtheorem*{notas}{Notations}

\theoremstyle{remark}
\newtheorem*{rem}{Remark}

\newcounter{envcount}
\NewEnviron{env}[2]{\refstepcounter{envcount}\text{#1 \textbf{\theenvcount}. #2}\hfill

\fbox{\begin{minipage}[t]{0.9\linewidth}
\BODY
\end{minipage}}}

\NewEnviron{env3}[2]{\refstepcounter{envcount}\text{#1 \textbf{\theenvcount}. #2}\hfill

\begin{minipage}[t]{0.9\linewidth}
\BODY
\end{minipage}}
\numberwithin{envcount}{section}

\NewEnviron{env2}[1]{#1

\fbox{\begin{minipage}[t]{0.9\linewidth}
\BODY
\end{minipage}}}

\NewEnviron{envalign}{
{\centering
$\displaystyle
\begin{aligned}
\BODY
\end{aligned}
$
\par}}

\newcounter{changebar}
\newcommand{\changebarcolor}{black}
\newcommand{\changestart}[1][black]{%
    \renewcommand{\changebarcolor}{#1}%
    \stepcounter{changebar}%
    \tikzmarknode{chbar-\thechangebar-start}{\strut}%
}
\newcommand{\changeend}{%
    \tikzmarknode{chbar-\thechangebar-end}{\strut}%
    \begin{tikzpicture}[remember picture, overlay]
        \draw[very thick, \changebarcolor] ([xshift={\oddsidemargin+1in-10pt}]current page.west |- chbar-\thechangebar-start.north) -- ([xshift={\oddsidemargin+1in-10pt}]current page.west |- chbar-\thechangebar-end.south);
    \end{tikzpicture}%
}

\setcounter{page}{1}
\counterwithin{equation}{section}

\maketitle
\begin{abstract}
    This article is devoted to the inverse source problem of uniquely determining a measure-valued source from sparse boundary measurements. The measurements considered consist of flux observations over a time interval at two distinct points on the boundary of the domain. The main objective of this work is to extend the existing literature on inverse source problems from sparse boundary measurements, which has so far been limited to point sources or $L^2$ sources, to the identification of a general class of Radon measures. Our approach combines several analytical tools, including regularity properties, boundary representations, and the time analyticity of solutions to the diffusion equation with singular sources. Our theoretical analysis is complemented by a numerical study of the problem. In particular, we investigate the reconstruction of point sources and of a source supported on a curve, and present numerical experiments illustrating the recovery of such sources from sparse boundary flux measurements.
\end{abstract}

\section{Introduction}

Let $\Omega = \mathbb{D}$ be the unit disc of $\R^2$, $T>0$, $Q := (0,T)\times\Omega$ the spacetime domain, $\Sigma := (0,T)\times\partial\Omega$ its lateral boundary. Consider a diffusion process on $\Omega$ described by the following initial boundary value problem 
\begin{align}\label{eqmu}
    \left\lbrace
    \begin{array}{rcll}
        \partial_tu - \Delta u &=& F(t,x),\quad &\text{in } Q,\\
        u(t,x) &=& 0,\quad &\text{on }\Sigma,\\
        u(0,x) &=& 0,\quad &\text{on }\Omega.
    \end{array}
    \right.
\end{align}
In the present article, we assume that the source term $F$ is a general measure-valued source with separated variables, namely,
$F(t,x)=\sigma(t)\mu(x)$, $t\in(0,T)$ and $x\in\Omega$,
where $\mu\in\left(\mathcal{C}(\overline{\Omega})\right)'$ is a finite signed Radon measure, that is, a regular Borel measure (see, e.g., \cite[Theorem 6.19]{rud} for more details). Under this assumption, we investigate the following inverse source problem :
\begin{itemize}
\item[{\bf(ISP)}] {\em Determine uniquely the measure-valued source $\mu$ from flux measurement, on two boundary points $x_1,x_2\in\partial\Omega$ over the time interval $(0,T)$, associated with the diffusion process $u$ solving \eqref{eqmu}.}
\end{itemize}
Recall that problem {\bf(ISP)} is motivated by a variety of applications involving diffusion processes generated by measure-valued sources. These include environmental applications, where numerous models have been developed to study the effects of pollution sources supported on lower-dimensional subsets of the domain. For instance, roadway air pollution may interfere with the assessment of the impact of a point source on air quality \cite{Nanni2022puff_lagrangianmodel_coastal_site}. In this context, Gaussian plume models for line sources have been developed \cite{Briant2011improved_line_source_air_pollution}. Diffusive underground pollution sources have also been investigated in large urban areas \cite{Chatterjee2020groundwater_cities_finite_sources}, as well as line contaminant sources in groundwater conductivity fields \cite{Zheng2024line_contaminant_sources_groundwater}. In these settings, problem {\bf(ISP)} provides a mathematical framework for quantifying, predicting, and preventing various pollution phenomena.

Motivated by its broad range of applications, problems similar to {\bf(ISP)} have attracted considerable attention from the mathematical community (see, e.g., \cite{Is1} for an overview). Most of the existing literature is devoted to the determination of sources from either internal or boundary measurements, assuming that the solution is known on an open subset of the domain $\Omega$ or of its boundary $\partial\Omega$. Without aiming to be exhaustive, we mention the works \cite{Choulli2006stability_estimates,IY}, which establish stable determination of $L^2$ sources by means of Carleman estimates. For the particular class of cylindrical domains, \cite{CJKZ,JKZ1} investigated both the theoretical and numerical reconstruction of such sources using contraction arguments and derived explicit error estimates for the proposed reconstruction schemes (see also \cite{JKZ} for results with final time observations). Concerning uniqueness, a general analysis of this problem, together with its limitations illustrated by explicit counterexamples, can be found in \cite{KLY,KSXY}. Motivated by applications to pollution source detection, the works \cite{Andrle2012moving_pollution_sources,ElBadiaHaDuong2002pollution_detection_problem} established uniqueness results for the recovery of point sources using techniques based primarily on unique continuation principles for parabolic equations. More recently, these uniqueness results have been strengthened by the stability estimates obtained in \cite{HJKT}
and a non-iterative method  to recover locations of the point sources by \cite{SuWa}.

Despite their generality, the above-mentioned results are costly in terms of the required measurements and are sometimes difficult to implement in practical applications. As an alternative, several authors have investigated inverse source problems from sparse measurements, where the observations are limited to a finite number of points, similarly to the formulation of problem {\bf(ISP)}. One of the earliest works in this direction is \cite{HR}, where {\bf(ISP)} was considered with $\sigma$ constant and $\mu$ given by the characteristic function of a domain. The analysis of \cite{HR} was subsequently extended in \cite{Run2020Heatequation_source_problem}, where {\bf(ISP)} was studied with $\sigma$ being a piecewise constant function to be identified, together with  a more general source term $\mu\in H^s(\Omega)$, $s\in(0,1/2)$. The approach developed in \cite{Run2020Heatequation_source_problem} was further extended in \cite{Li2020fractionalorder_sourceterm} to a broader class of time-fractional diffusion equations. More recently, {\bf(ISP)} with a single point source (i.e., $\mu=\delta_y$, where $\delta_y$ denotes the Dirac delta measure at $y\in\Omega$) was investigated in \cite{GZZ} under the assumption that $\sigma\equiv1$. The results of \cite{GZZ} were significantly extended in \cite{gong2026identification}, where the authors considered {\bf(ISP)} with a single point source while simultaneously recovering a general time-dependent coefficient $\sigma$. In addition, \cite{gong2026identification} also analyzed {\bf(ISP)} with a single point source in higher spatial dimensions when $\sigma$ is piecewise constant. We also mention the works of \cite{JET,MGH,WKT}, which addressed a similar class of inverse source problems for hyperbolic equations. To the best of our knowledge, despite its strong physical motivation, the resolution of {\bf(ISP)} when $\mu$ belongs to the general class of finite signed Radon measures remains an open problem. Moreover, even the recovery of a source term $\mu\in H^s(\Omega)$, $s\in(0,1/2)$, in the case where $\sigma$ is not piecewise constant does not appear to have been addressed in the existing literature. These two extensions of the current theory for {\bf(ISP)} constitute the main objectives of the present paper.
 

Let us give a more precise formulation of problem {\bf(ISP)}.
To this end, we recall that $u$ denotes the solution in the transposition sense of \eqref{eqmu} (see Definition \ref{d1}) and we assume that
$\supp\mu\subset\subset\Omega$, where, for any set $\omega$, the notation
$\omega\subset\subset\Omega$
means that $\overline{\omega}\subset\Omega$, and that
$\sigma\in L^2(0,T)$.
Propositions~\ref{IBVPsolution} and~\ref{p1} show that
$u\in L^2(Q)$
and that there exists
$r\in(0,1)$
such that $u|_{(0,T)\times\overline{C_r}}
\in L^2\bigl(0,T;\mathcal C^1(\overline{C_r})\bigr)$,
where
$C_r:=\left\{x\in\R^2\,;\, r<|x|_2<1\right\}$
and $|\cdot|_2$ denotes the Euclidean norm in $\R^2$.
We also denote by $\nu$ the outward unit normal vector to $\partial\Omega$.
We can now reformulate problem {\bf(ISP)} as follows:
\begin{itemize}
\item[{\bf(ISP')}] {\em
Fix $x_1,x_2\in\partial\Omega$ and assume that $\sigma$ is known.
Determine uniquely the measure-valued source $\mu$ from the knowledge of
$\partial_\nu u(t,x_\ell)$, a.e. 
$t\in(0,T)$,
$\ell=1,2$.
}
\end{itemize}

Our main result provides the following positive answer to problem {\bf(ISP')}.

\begin{thm}\label{inv_prob}
    Set $\{x_\ell = e^{i\theta_\ell}\}_{\ell=1,2}\subset\partial\Omega$ satisfying the assumption $\theta_1 - \theta_2\notin\pi\Q$ and, for $j=1,2$, let $\mu_j$ be finite signed Radon measures such that $\supp{\mu_j}\subset\subset\Omega$.
    Let $\sigma\in H^1(0,T)$ be a non-uniformly vanishing function such that  $\sigma $ is constant on $(T_1,T)$, for some $T_1\in(0,T)$, and $\sigma(0) = 0$. We denote by $u_j$ the solution in the transposition sense of \eqref{eqmu} with $\mu = \mu_j$. Then the condition
    \begin{align}\label{Hypothesis_normalderivative}
        \partial_{\nu}u_1(t,x_\ell) = \partial_{\nu}u_2(t,x_\ell)\,,\;  \textrm{a.e. }t\in(0,T),\;  \ell= 1,2
    \end{align}
    implies that $\mu_1 = \mu_2$.
  
\end{thm}

To the best of our knowledge, Theorem~\ref{inv_prob} provides the first positive answer to problem {\bf(ISP')} for a general  finite signed Radon measure $\mu$ with boundary observation points
$\{x_\ell = e^{i\theta_\ell}\}_{\ell=1,2}\subset\partial\Omega$
only subjected to the generic condition
$\theta_1-\theta_2\notin\pi\Q$,
already considered in
\cite{GZZ,Run2020Heatequation_source_problem}.
As mentioned above, beyond its intrinsic mathematical interest, this extension of the existing literature on problem {\bf(ISP')} to the general class of Radon measures is motivated by various applications that cannot be modelled either by point sources or by $L^2$ sources.
One of the main challenges in this extension stems from the low regularity of the source measure $\mu$.
We overcome this difficulty by deriving the boundary representation \eqref{bound-rep} of the flux $\partial_\nu u$, involving the solution of the elliptic problem \eqref{fo3} together with a Fourier series expansion based on suitably chosen eigenfunctions of the Laplacian.
Our analysis combines this representation formula with regularity and analyticity properties of problem \eqref{eqmu}, as well as techniques from complex and harmonic analysis.

In contrast to
\cite{GZZ,Run2020Heatequation_source_problem},
where the analysis was restricted to piecewise constant time-dependent source terms $\sigma$, we consider here a general function
$\sigma\in H^1(0,T)$
subject only to the conditions
$\sigma$ constant on $(T_1,T]$ and $\sigma(0)=0$.
This regularity assumption was previously relaxed only in
\cite{gong2026identification},
where the authors also addressed the simultaneous recovery of $\sigma$.
However, their approach relies on a reduction to an elliptic inverse source problem that applies only to the case of a single point source. Moreover, due to the lack of regularity for more general classes of signed Radon measures, it remains unclear whether either of the conditions $\sigma\in H^1(0,T)$ or $\sigma(0)=0$ can be removed.
Even in the case
$\mu_j\in L^2(\Omega)$, $j=1,2$,
Theorem~\ref{inv_prob} appears to be new.
Moreover, when
$\mu_j\in L^2(\Omega)$,
the assumption
$\sigma(0)=0$
can be removed, while for
$\mu_j\in H^s(\Omega)$, $s\in(0,1/2)$
our analysis can be carried out under the weaker assumption
$\sigma\in L^2(0,T)$. We refer Corollary \ref{c1} in the appendix, for the precise statement of our answer to problem {\bf(ISP')} with such class of more regular source $\mu$.

The assumption that $\sigma$ is constant on $(T_1,T]$ is an important technical requirement. It has been imposed in all the works on problem {\bf(ISP')} known to us (see, e.g.,
\cite{gong2026identification,GZZ,Run2020Heatequation_source_problem}),
and it is not clear whether it can be relaxed.
A similar remark applies to the restriction of the domain to the unit disc.
Indeed, as in \cite{gong2026identification,GZZ,Run2020Heatequation_source_problem},
our analysis relies on explicit properties of the eigenfunctions of the Laplacian.

In Section \ref{sec:num} we present numerical results to illustrate the feasibility of recovering point sources with known constant amplitudes and a line source, which is supported on a curve, from two point flux boundary measurements.

This article is organized as follows. In Section~\ref{sec:analysis}, we present preliminary results on problem~\eqref{eqmu}, including a representation formula for the solution and the time analyticity of the boundary flux. The proof of the main result is given in Section~\ref{sectionproof}. In Section~\ref{sec:num}, we present a numerical analysis of {\bf(ISP')}, investigating the reconstruction of point sources and of a source supported on a curve, and providing numerical experiments illustrating the recovery of such sources from sparse boundary flux measurements. Finally, in the Appendix, we  recall some useful properties of finite signed Radon measures stated in Lemma~\ref{nullmeasure}, and state and prove Corollary~\ref{c1}, where problem {\bf(ISP')} is considered for $\mu\in H^s(\Omega)$ with $s\in[0,1/2)$. 

\section{Analysis of the direct problem}\label{sec:analysis}

 In this section we collect some preliminary useful properties of solutions of \eqref{eqmu} that we consider in the transposition sense. We start by proving unique existence of such solutions. Then we consider increasing regularity properties away from the support of the source.

\subsection{Solutions in the transposition sense}

 Recall that, by the Sobolev embedding theorem we have $H^2(\Omega)\hookrightarrow\mathcal{C}(\overline{\Omega})$. Consequently, finite signed Radon measures are continuous linear forms on $H^2(\Omega)$ and integration with respect to such measures is defined as a consequence of the Radon-Nikodym theorem (see \cite[Remark 6.18, pp. 129]{rud}). Following \cite[Chapter~3~Section~2]{lionsmagenesvol1}, we can define solutions in the transposition sense of \eqref{eqmu} as follows.\bigskip

\begin{df}\label{d1}
We say that $u\in L^2(Q)$ is a solution in the sense of transposition to \eqref{eqmu}  if for every $w\in L^2(0,T,H^2(\Omega)\cap H_0^1(\Omega))\cap H^1(0,T, L^2(\Omega))$, satisfying $w(T,x) = 0$ for a.e. $x\in\Omega$, we have
\begin{align}
    \int_0^T\int_{\Omega} u(t,x)\left(-\partial_tw(t,x) - \Delta w(t,x)\right) dxdt = \int_0^T\int_{\Omega}\sigma(t) w(t,x)d\mu(x)dt.
\end{align}
\end{df}
Using this definition, we obtain the following result.
\begin{prop}\label{IBVPsolution}
The initial boundary value problem \eqref{eqmu} has a unique solution $u\in L^2(Q_T)$ in the transposition sense. Moreover, there exists a constant $C>0$ depending only on $T$ and $\Omega$ such that
$\norm{u}_{L^2(Q)} \leq C\abs{\mu}(\Omega)\norm{\sigma}_{L^2(0,T)}$,
where $|\mu|$ denotes the total variation of $\mu$.
\end{prop}
\begin{proof}
By \cite[Chapter~4,~Theorem~1.1]{lionsmagenesvol2} for any $F\in L^2(Q)$ there exists a unique $w_F\in L^2(0,T,H^2(\Omega)\cap H_0^1(\Omega))\cap H^1(0,T,L^2(\Omega))$ solving the following problem
\begin{align}\label{problemF}
    \left\lbrace
    \begin{array}{l}
        -\partial_tw_F(t,x) - \Delta w_F(t,x) = F(t,x)\quad\text{in } Q_T,\\
        w_F = 0\quad\text{on }\Sigma_T,\\
        w(T,\cdot) = 0\quad \text{on }\Omega
    \end{array}
    \right.
\end{align}
and satisfying, for a constant $C>0$ depending only on $T$ and $\Omega$, the estimate
\begin{align}\label{estimate_wF}
    \norm{w_F}_{L^2(0,T,H^2(\Omega))} + \norm{w_F}_{H^1(0,T,L^2(\Omega))}\leq C\norm{F}_{L^2(Q)}.
\end{align}
We define the map
\begin{align*}
    \zeta : F\in L^2(Q)\mapsto\int_0^T\int_{\Omega}\sigma(t) w_F(t,x)d\mu(x)dt.
\end{align*}
Applying the embedding $H^2(\Omega)\hookrightarrow \mathcal{C}(\overline{\Omega})$, the estimate \eqref{estimate_wF} and the fact that $\mu$ is a finite signed Radon measure, we find
\begin{align*}
    \abs{\zeta(F)} \leq \int_0^T\abs{\sigma(t)}\;\norm{w_F(T-t,\cdot)}_{\mathcal C(\overline{\Omega})}\abs{\mu}(\Omega)dt
    &\leq \abs{\mu}(\Omega)\norm{\sigma}_{L^1(0,T)}\norm{w_F}_{L^1(0,T,\mathcal{C}(\overline{\Omega}))}\\
    &\leq C\abs{\mu}(\Omega)\norm{\sigma}_{L^2(0,T)}\norm{w_F}_{L^2(0,T,H^2(\Omega))}\\
    &\leq C\abs{\mu}(\Omega)\norm{\sigma}_{L^2(0,T)}\norm{F}_{L^2(Q)},
\end{align*}

\noindent 
where $C>0$ is a constant depending only on $\Omega$ and $T$ that might change from line to line. Thus $\zeta$ is a continuous linear form on $L^2(Q)$ with norm $\norm{\zeta}\leq C\abs{\mu}(\Omega)\norm{\sigma}_{L^2(0,T)}$. By the Riesz representation theorem, there exists a unique $u\in L^2(Q)$ such that
\begin{align*}
    \scalarprod{u,F}_{L^2(Q)} = \zeta(F) = \int_Q\sigma(t)w_F(T-t,x)d\mu(x)dt,\quad \forall F\in L^2(Q),
\end{align*}

\noindent and $\norm{u}_{L^2(Q)}\leq C\abs{\mu}(\Omega)\norm{\sigma}_{L^2(0,T)}$. Then $u$ is the unique solution in the transposition sense to \eqref{eqmu}.
\end{proof}

\subsection{Regularity of the solution near the boundary}

 \bigskip

\begin{prop}\label{p1}
There exists $r\in(0,1)$, such that  the solution $u$ in the transposition sense of \eqref{eqmu}, with $\supp\mu\subset\subset\Omega$, satisfies $u|_{(0,T)\times \overline{C_r}}\in L^2(0,T,\mathcal C^1(\overline{C_r}))$ .
\end{prop}

\begin{proof}

Since $\supp\mu\subset\subset\Omega$ there exists $0<r_0<1$ such that $\supp\mu\subset\mathcal{B}_{r_0}$ where $\mathcal{B}_{r_0} := \left\lbrace x\in\R^2\; ;\; \abs{x}_2< r_0\right\rbrace$. Let $V$ be  an open neighbourhood of the boundary $\partial\Omega = \mathbb{S}^1$ with $\overline{V}\cap\overline{\mathcal{B}_{r_0}} = \emptyset$ and $\chi\in\mathcal{C}^{\infty}(\R^2)$ such that $\chi = 1$ on neighbourhood of $\overline{V}$ and $\chi=0$ on neighbourhood of $\overline{\mathcal{B}_{r_0}}$. We fix $v\in L^2(Q)$ defined by $v(t,x) = \chi(x) u(t,x)$, $(t,x)\in Q$, with $u\in L^2(Q)$ the solution in the transposition sense of \eqref{eqmu}.  Fix also  $\mathcal H_T:=\{w\in H^1(0,T;L^2(\Omega))\cap L^2(0,T;H^2(\Omega)\cap H^1_0(\Omega)):\ w(T,\cdot)\equiv0\}$. Observing that $\chi = 0$ on a neighbourhood of $\supp\mu$ and  $\nabla u\in L^2(0,T,H^{-1}(\Omega))$, for all $w\in\mathcal H_T$, we obtain 
\begin{align*}
    \int_Qv(-\partial_tw-\Delta w)dxdt &= \int_0^T\int_\Omega u(-\partial_t(\chi(x)w)-\Delta(\chi(x)w)+2\nabla\chi\cdot\nabla w+(\Delta\chi)w)(t,x)dxdt\\
    &=\int_0^T\int_{\Omega}\chi(x)\sigma(t)w(t,x)d\mu(x)dt-2\int_0^T\left\langle \nabla\chi\cdot\nabla u(t,\cdot),w(t,\cdot)\right\rangle_{H^{-1}(\Omega),H^1_0(\Omega)}dt\\
    &\ \ \ - \int_0^T\int_{\Omega}(\Delta\chi)uwdxdt\\
    &= \left\langle(-2\nabla\chi\cdot\nabla u - (\Delta\chi)u),w\right\rangle_{L^2(0,T;H^{-1}(\Omega)),L^2(0,T;H^1_0(\Omega))}. 
\end{align*}
Therefore, we have  $\partial_tv-\Delta v\in L^2(0,T,H^{-1}(\Omega))$ and applying \cite[Chapter~3~Example~4.7.1]{lionsmagenesvol1}, we obtain $\chi u=v \in L^2(0,T,H_0^1(\Omega))\cap \mathcal{C}([0,T],L^2(\Omega))$.\bigskip

Similarly, we fix $U$ a neighbourhood  of $\partial\Omega$ such that $U\subset V\subset\supp\chi$, $\chi_2\in C^\infty(\overline{\Omega})$ such that  $\chi_2 = 1$ on $U$ and $\chi_2=0$ on a neighbourhood of $\overline{\Omega}\setminus V$, and  we set $v_2\in L^2(Q)$ defined by $v_2(t,x) = \chi_2(x) v(t,x)$, $(t,x)\in Q$.  Since $\chi=1$ on $\supp{\chi_2}$, we have $\chi_2\chi = \chi_2$ so $v_2 = \chi_2 v = \chi_2\chi u = \chi_2 u$. Then, similarly as above, for all $w\in \mathcal H_T$, we obtain
\begin{align*}
  \int_Qv_2(-\partial_tw-\Delta w)dxdt &= \left\langle\underbrace{\chi_2(-2\nabla\chi\cdot\nabla u - (\Delta\chi)u) }_{=0}-2\nabla\chi_2\cdot\nabla v - (\Delta\chi_2)v,w\right\rangle_{L^2(0,T;H^{-1}(\Omega)),L^2(0,T;H^1_0(\Omega))}\\
    &= \left\langle(-2\nabla\chi_2\cdot\nabla v - (\Delta\chi_2)v),w\right\rangle_{L^2(0,T;H^{-1}(\Omega)),L^2(0,T;H^1_0(\Omega))}.
\end{align*}

Thus, we have $\partial_tv_2-\Delta v_2\in L^2(Q)$ and, from \cite[Chapter~4~Theorem~1.1]{lionsmagenesvol2}, we obtain $v_2 = \chi_2 u\in L^2(0,T,H^2(\Omega)\cap H_0^1(\Omega))\cap H^1(0,T,L^2(\Omega))$. Combining this with the fact that $v_2 = u$ on $(0,T)\times U$, we get $u|_{(0,T)\times U}\in L^2(0,T,H^2(U))$. Repeating this process combined with \cite[Proposition 2.3, Chapter 4]{lionsmagenesvol2}, \cite[Theorem 5.3, Chapter 4]{lionsmagenesvol2}, we deduce that there exists  $U_1$ a neighbourhood  of $\partial\Omega$, with smooth boundary, such that $u|_{(0,T)\times U_1}\in L^2(0,T,H^3(U_1))$ and by the Sobolev embedding theorem we get $H^3(U_1)\hookrightarrow\mathcal{C}^1(\overline{U_1})$ which implies that 
$u|_{(0,T)\times \overline{U_1}}\in L^2(0,T,\mathcal{C}^1(\overline{U_1}))$.
In addition, we can find $r\in(r_0,1)$ such that $C_r\subset U$ and $u|_{(0,T)\times \overline{C_r}}\in L^2(0,T,\mathcal{C}^1(\overline{C_r}))$.
\end{proof}

\subsection{Decomposition in Fourier series of the solution}
From now on,  we denote by $\left\langle\cdot,\cdot \right\rangle$ the inner product on $L^2(\Omega)$ defined by
$$\left\langle f,g \right\rangle=\int_\Omega f\overline{g}dx,\quad f,g\in L^2(\Omega).$$
 We recall that the operator $A = -\Delta$ acting on $L^2(\Omega)$ with domain $D(A) = H^2(\Omega)\cap H_0^1(\Omega)$  admits an eigen-system $\{(\lambda_n,\phi_n)\}_{n\geq 1}$, with eigenvalues ordered in increasing order $0<\lambda_1\leq \lambda_2\leq\ldots\:$, $\lambda_n\rightarrow +\infty$ as $n\rightarrow \infty$ and eigenfunctions $\phi_n\in\mathcal{C}^{\infty}(\overline{\Omega})\cap H_0^1(\Omega)$ that form an orthonormal basis of $L^2(\Omega)$. Let us recall the definition of   the space $X^s=D(A^{\frac{s}{2}})$ for $s\in\R$.

\begin{df}\label{frac_order_sobolev_spaces}
    For $s\geq 0$, the Hilbert space $X^s$ is defined by\newline
    $X^s := \left\lbrace u = \sum\limits_{n=1}^{+\infty}\scalarprod{u,\phi_n}\phi_n\in L^2(\Omega)\,\bigg|\, \sum\limits_{n = 1}^{+\infty}\lambda_n^s\abs{\scalarprod{u,\phi_n}}^2 < +\infty\right\rbrace$, with the norm $\norm{u}_{s} := \left(\sum\limits_{n = 1}^{+\infty}\lambda_n^s\abs{\scalarprod{u,\phi_n}}^2\right)^{\frac12}$, for $u\in X^s$, and $X^{-s}$ is the dual space of $X^s$ with the norm $\norm{u}_{-s} := \left(\sum\limits_{n = 1}^{+\infty}\lambda_n^{-s}\abs{\scalarprod{u,\phi_n}_{-\frac{s}{2},\frac{s}{2}}}^2\right)^{\frac12}$, for $u\in X^{-s}$, where $\scalarprod{u,\phi}_{-\frac{s}{2},\frac{s}{2}} := \scalarprod{u,\phi}_{X^{-s}, X^s}$, for $\phi\in X^s$.
\end{df}
The series of continuous embeddings $X^{1+\delta}\subset H^{1+\delta}(\Omega)\subset\mathcal{C}(\overline{\Omega})$, for all $\delta>0$, implies that
\begin{equation}\label{mureg}\mu\in\bigcap_{\delta>0} X^{-1-\delta}.\end{equation}
    From now on, for $\delta>0$, $\psi\in X^{1+\delta}$, $\scalarprod{\mu,\psi}_{-\frac{1+\delta}{2},\frac{1+\delta}{2}}$ will be simply denoted by $\scalarprod{\mu,\psi}$.
    
 Using these properties, we prove the following important decomposition in Fourier series of the solution $u\in L^2(Q)$ in the transposition sense of \eqref{eqmu}.

\begin{thm}\label{Fourierseriesofu}
   The solution $u\in L^2(Q)$ in the transposition sense of \eqref{eqmu} is expressed as the following Fourier series :
    \begin{align}
        u(t,x) &\label{fo1}= \sigma(t)\sum_{n = 1}^{+\infty}\frac{\scalarprod{\mu,\phi_n}}{\lambda_n}\phi_n - \sum_{n=1}^{+\infty}\left(\int_0^te^{-\lambda_n(t-s)}\sigma'(s)\frac{\scalarprod{\mu,\phi_n}}{\lambda_n}ds\right)\phi_n\\
        &\label{fo2}= \sigma(t)G(x) + H(t,x),\quad (t,x)\in Q.
    \end{align}
    Moreover, the series $H$ converges in $L^2(0,T,H^{3-\delta}(\Omega)\cap H_0^1(\Omega))$ and $G\in H^{1-\delta}(\Omega)$   solves in the transposition sense the following boundary value problem :
    \begin{align}\label{fo3}
        \left\lbrace
        \begin{array}{ll}
            -\Delta G = \mu\quad\text{in }\Omega\\
            G\mid_{\partial\Omega} = 0
        \end{array}
        \right.
    \end{align}
Finally, there exists $r\in(0,1)$, such that $G|_{C_r}\in H^3(C_r)\hookrightarrow\mathcal{C}^1(\overline{C_r})$.
\end{thm}

In order to prove Theorem \ref{Fourierseriesofu}, we need three intermediate results. 

\begin{lem}\label{regG}
    The function  $G$ defined by \eqref{fo2} is lying in  $ H^{1-\varepsilon}(\Omega)$, for every $\varepsilon\in (0,1)$.
\end{lem}

\begin{proof} Let  $\delta\in (0,\frac{1}{2})$. In view of \eqref{mureg}, we have $\mu\in X^{-1-\delta}$ and it follows
    \begin{align*}
        \sum_{n=1}^{+\infty}\lambda_n^{1-\delta}\frac{\abs{\scalarprod{\mu,\phi_n}}^2}{\lambda_n^2} = \sum_{n=1}^{+\infty}\lambda_n^{-1-\delta}\abs{\scalarprod{\mu,\phi_n}}^2 <+\infty,
    \end{align*}
     which implies that $G\in H^{1-\delta}(\Omega)$.
\end{proof}

\begin{lem}\label{regH}
    The function  $H$ defined by \eqref{fo2} is lying in $L^2(0,T;X^{3-\delta})\cap H^1(0,T;X^{1-\delta})$, for all $\delta>0$.
\end{lem}

\begin{proof}
    Fix $n\in\N :=\{1,2,\ldots\}$ and denote $H_n$ the map defined by
    \begin{align}
        H_n(t) := -\int_0^te^{-\lambda_n(t-s)}\sigma'(s)\frac{\scalarprod{\mu,\phi_n}}{\lambda_n}ds,\quad \forall t\in[0,T].
    \end{align}
    Since $\sigma\in H^1(0,T)$ we get $H_n\in H^1(0,T)$ with $
        H_n'(t) = -\lambda_nH_n(t)-\sigma'(t)\frac{\scalarprod{\mu,\phi_n}}{\lambda_n}$, $t\in(0,T)$.
    Let us estimate $H_n(t)$ and $H_n'(t)$. Applying Young's inequality for convolution product,  we get
    \begin{align}
        \norm{H_n}_{L^2(0,T)}&\leq \norm{\sigma'}_{L^2(0,T)}\left(\int_0^{+\infty}e^{-\lambda_ns}ds\right)\frac{\left|\scalarprod{\mu,\phi_n}\right|}{\lambda_n}\leq \norm{\sigma'}_{L^2(0,T)}\frac{\left|\scalarprod{\mu,\phi_n}\right|}{\lambda_n^2},\label{Hn}\\
        \norm{H_n'}_{L^2(0,T)}&\leq \norm{\sigma'}_{L^2(0,T)}\frac{\left|\scalarprod{\mu,\phi_n}\right|}{\lambda_n} + \lambda_n\norm{\sigma'}_{L^2(0,T)}\frac{\left|\scalarprod{\mu,\phi_n}\right|}{\lambda_n^2}\leq 2\norm{\sigma'}_{L^2(0,T)}\frac{\left|\scalarprod{\mu,\phi_n}\right|}{\lambda_n}.\label{Hn'}
    \end{align}

     Let us prove that $H\in L^2(0,T,X^{3-\delta})$ which follows from the convergence in $L^2(0,T,X^{3-\delta})$ of the series $\sum\limits_{n=1}^{+\infty}H_n(t)\phi_n$. Set $n,m\in\N$ such that $n< m$ and $0<\delta<3$. Applying Fubini theorem combined with \eqref{Hn}, we get
    \begin{align*}
        \norm{\sum_{k=n+1}^mH_k(t)\phi_k}^2_{L^2(0,T;X^{3-\delta})}&\leq \sum_{k=n+1}^m\int_0^T\lambda_k^{3-\delta}\abs{H_k(t)}^2\\
        &\leq \sum_{k=n+1}^m\lambda_k^{3-\delta}\norm{\sigma'}_{L^2(0,T)}^2\frac{\left|\scalarprod{\mu,\phi_n}\right|^2}{\lambda_n^4}\\
        &\leq \norm{\sigma'}_{L^2(0,T)}^2\sum_{k=n+1}^m\lambda_k^{-1-\delta}\left|\scalarprod{\mu,\phi_n}\right|^2\xrightarrow[n,m\longrightarrow+\infty]{}0.
    \end{align*}
 Thus,  the series $\sum\limits_{n=1}^{+\infty}H_n(t)\phi_n$ is  a Cauchy  sequence in $L^2(0,T,X^{3-\delta})$, which means that it converges in $L^2(0,T,X^{3-\delta})$. It follows that $H\in L^2(0,T,X^{3-\delta})$.
Similarly, using \eqref{Hn'}, we show that $H\in H^1(0,T;X^{1-\delta})$.
\end{proof}
\begin{lem}\label{solve_distribution} 
Let $u_n$, $n\in\mathbb N$, be defined by 
$u_n(t) = \scalarprod{u(t,\cdot),\phi_n}$ for almost every $t\in(0,T)$. Then, we have 
\begin{align}\label{so1}
u_n(t) = \int_0^te^{-\lambda_n(t-s)}\sigma(s)ds\scalarprod{\mu,\phi_n}, \quad t\in(0,T).
\end{align}
\end{lem}

\begin{proof} Fix $\psi\in C^\infty_0(0,T)$, and consider $w(t,x)=\psi(t)\phi_n(x)$, $(t,x)\in Q$.
Recalling that $u$ is the solution in the transposition sense to \eqref{eqmu} and $w\in C^\infty_0(0,T;H^2(\Omega)\cap H^1_0(\Omega))$,
we find
$$\begin{aligned}-\int_0^Tu_n(t)\psi'(t)dt+\lambda_n\int_0^Tu_n(t)\psi'(t)dt=\int_Qu(-\partial_tw-\Delta w)dxdt&=\int_0^T\int_\Omega \sigma(t)w(t,x)\mu(x)dt\\
&=\int_0^T\int_\Omega \sigma(t)\phi_n(x)\psi(t)d\mu(x)dt.\end{aligned}$$
Thus, $u_n$ solves in the sense of distribution in $(0,T)$ the ordinary differential equation $u_n' + \lambda_nu_n = \sigma\scalarprod{\mu,\phi_n}$. This implies that there exists $a_n\in R$ such that 
\begin{align}\label{so2}u_n(t) =a_ne^{-\lambda_n t}+ \int_0^te^{-\lambda_n(t-s)}\sigma(s)ds\scalarprod{\mu,\phi_n}.\end{align}
Therefore, the proof of \eqref{so1} will be completed if we show that $a_n=0$. In view of \eqref{so2} and the fact that $\sigma\in H^1(0,T)$, it is clear that $u_n\in H^1(0,T)$, $u_n(0)=a_n$ and, for almost every $t\in(0,T)$, we have
\begin{align}\label{so3}u_n'(t) + \lambda_nu_n(t) = \sigma(t)\scalarprod{\mu,\phi_n}.\end{align}
 Choosing $\psi\in C^1([0,T])$ with $\psi(T)=0$, $\psi(0)=-1$,  $w(t,x)=\psi(t)\phi_n(x)$, $(t,x)\in Q$, and  repeating the above argumentation combined with \eqref{so3}, we find
$$\begin{aligned}
\int_0^T\sigma(t)\scalarprod{\mu,\phi_n}\psi(t)dt+a_n&=\int_0^T(u_n'+\lambda_nu_n)\psi(t)dt-u_n(0)\psi(0)\\
&=-\int_0^Tu_n(t)\psi'(t)dt+\lambda_n\int_0^Tu_n(t)\psi(t)dt\\
&=\int_Qu(-\partial_tw-\Delta w)dxdt
=\int_0^T\sigma(t)\scalarprod{\mu,\phi_n}\psi(t)dt.\end{aligned}$$
This proves that $a_n=0$ and we obtain \eqref{so1} from \eqref{so3}.

\end{proof}
 Armed with Lemma \ref{regG}, \ref{regH} and \ref{solve_distribution} we are now in position to complete the proof of Theorem \ref{Fourierseriesofu}.
\begin{proof}[\bf{Proof of Theorem \ref{Fourierseriesofu}}]
    Applying Lemma \ref{solve_distribution}, for almost every $t\in(0,T)$, we have
    \begin{align}\label{pf3}u(t,\cdot)=\sum_{n=1}^{+\infty}u_n(t)\phi_n,\quad 
        u_n(t) = \int_0^te^{-\lambda_n(t-s)}\sigma(s)\scalarprod{\mu,\phi_n}ds.
    \end{align}
    Recalling that $\sigma\in H^1(0,T)$, with $\sigma(0) = 0$, and integrating by parts, we obtain
    \begin{align*}
       \int_0^te^{-\lambda_n(t-s)}\sigma(s)\scalarprod{\mu,\phi_n}ds  &= \left[\frac{1}{\lambda_n}e^{-\lambda_n(t-s)}\sigma(s)\scalarprod{\mu,\phi_n} \right]_0^t - \int_0^te^{-\lambda_n(t-s)}\sigma'(s)\frac{\scalarprod{\mu,\phi_n}}{\lambda_n}ds\\
        &= \sigma(t)\frac{\scalarprod{\mu,\phi_n}}{\lambda_n} - \int_0^te^{-\lambda_n(t-s)}\sigma'(s)\frac{\scalarprod{\mu,\phi_n}}{\lambda_n}ds.
    \end{align*}
This proves \eqref{fo1} and, in view of Lemma \ref{regG}, \ref{regH}, we have
$H\in L^2(0,T,H^{3-\delta}(\Omega)\cap H_0^1(\Omega))$ and $G\in H^{1-\delta}(\Omega)$.

Let us  prove that $G$ solves \eqref{fo3} in the sense of transposition. 
For this purpose, fix $\psi\in H^2(\Omega)\cap H_0^1(\Omega)$ and observe that
    \begin{align*}
        \scalarprod{G,-\Delta\psi}_{L^2(\Omega)} = \sum_{n=1}^{+\infty}\frac{\scalarprod{\mu,\phi_n}}{\lambda_n}\scalarprod{\phi_n,-\Delta\psi}= \sum_{n=1}^{+\infty}\scalarprod{\mu,\phi_n}\scalarprod{\psi,\phi_n}.
    \end{align*}

    Now recalling that $\psi\in D(A) = H^2(\Omega)\cap H_0^1(\Omega)$ then the series $\sum\limits_{n=1}^{+\infty}\scalarprod{\psi,\phi_n}\phi_n$ converges in $D(A)$. Since $\mu\in X^{-2}$ we find
    \begin{align*}
        \scalarprod{G,-\Delta\psi}_{L^2(\Omega)} = \lim_{N\rightarrow+\infty}\sum_{n=1}^N\scalarprod{\mu,\phi_n}_{-1,1}\scalarprod{\psi,\phi_n}
        &= \lim_{N\rightarrow+\infty}\scalarprod{\mu,\sum_{n=1}^N\scalarprod{\psi,\phi_n}\phi_n}_{-1,1}\\
        &= \scalarprod{\mu,\psi}_{-1,1} = \int_{\Omega}\psi(x)d\mu(x).
    \end{align*}
This clearly proves that $G$ solves \eqref{fo3} in the sense of transposition.

In order to complete the proof of the theorem, we only need to show the last statement of the theorem. Similarly to Proposition \ref{p1}, we set $G_1(x) = \chi(x) G(x)$, $x\in\Omega$, where $\chi = 1$ on a neighbourhood of  $\overline{V}$ and $\chi=0$ on a neighbourhood of $\overline{\mathcal{B}_{r_0}}$. Then for all $\phi\in H^2(\Omega)\cap H_0^1(\Omega)$ we have $\chi\phi\in H^2(\Omega)\cap H_0^1(\Omega)$ 	and $\scalarprod{\mu,\chi\phi} = 0$ as $\chi =0$ on $\supp\mu$ so in the transposition sense  we have
    \begin{align*}
        0=\scalarprod{G,-\Delta(\chi\phi)}_{L^2(\Omega)}  &= \scalarprod{G(-\Delta\chi),\phi}_{L^2(\Omega)}  + \scalarprod{G_1,-\Delta\phi}_{L^2(\Omega)}  + \scalarprod{G,-2\nabla\chi\cdot\nabla\phi}_{L^2(\Omega)} \\
        &=\scalarprod{G_1,-\Delta\phi}_{L^2(\Omega)}- \scalarprod{((-\Delta\chi)G -2\nabla\chi\cdot\nabla G),\phi}_{H^{-1}(\Omega),H^1_0(\Omega)}.
    \end{align*}
  Therefore, the function $G_1$ solves in the transposition sense the problem :
   \begin{align}\label{eqtruncationG1}
   		\left\lbrace
   		\begin{array}{rcll}
   			-\Delta G_1 &=& (-\Delta\chi)G -2\nabla\chi\cdot\nabla G,\quad &\text{in }\Omega.\\
   			G_1 &=& 0,\quad &\text{on }\partial\Omega.
   		\end{array}
   		\right.
   \end{align}

	Then, by the Lax-Milgram theorem we deduce $G_1\in H_0^1(\Omega)$. Similarly, fixing $\chi_2\in C^\infty(\R^2)$, such that $\chi_2=0$ on a neighbourhood of $\Omega\setminus V$ and $\chi_2=1$ on a neighbourhood $U_1$ of $\partial\Omega$, and fixing  $G_2 = \chi_2 G_1$ we have $G_2 = \chi_2 G_1 = \chi_2\chi G = \chi_2 G$. Then, it follows
    \begin{align*}
        -\Delta G_2 = \chi_2(-\Delta G_1) -\Delta\chi_2 G_1 - 2\nabla\chi_2\cdot\nabla G_1= (-\Delta\chi_2) G_1 - 2\nabla\chi_2\cdot\nabla G_1,
    \end{align*}
    which proves that $G_2$ solves the equation $-\Delta G_2 = (-\Delta\chi_2) G_1 - 2\nabla\chi_2\cdot\nabla G_1$ in $\Omega$.
Recalling that $- 2\nabla\chi_2\cdot\nabla G_1\in L^2(\Omega)$ and applying \cite[Theorem 5.2, Chapter 2]{lionsmagenesvol1}, we deduce that $G_2\in H^2(\Omega)\cap H_0^1(\Omega)$. This implies that $G|_{U_1}=G_2|_{U_1}\in H^2(U_1)$. Repeating these arguments and applying \cite[Theorem 5.2, Chapter 2]{lionsmagenesvol1}, we can find a  neighbourhood $U_2$ of $\partial\Omega$, with smooth boundary, such that $G|_{U_2}\in H^3(U_1)$. Fixing $r\in(0,1)$ such that $C_r\subset U_2$ and applying  Sobolev embedding theorem, we deduce that $G|_{C_r}\in H^3(C_r)\hookrightarrow\mathcal{C}^1(\overline{C_r})$. This completes the proof of the theorem.
    
\end{proof}

\subsection{Boundary representation  of the solution and analyticity}
Let $u\in L^2(Q)$ be the solution in the transposition sense of \eqref{eqmu}. In view of Proposition \ref{p1}, we can define $\partial_\nu u|_{\Sigma}=\nabla u\cdot\nu|_{\Sigma}$ as an element of $L^2(0,T;\mathcal C(\partial\Omega))$. Combining this with  the Fourier series decomposition of $u$ stated in Theorem \ref{Fourierseriesofu}, we can derive the explicit  boundary flux representation.


\begin{prop}\label{bound}
     The series associated with $\partial_{\nu}H$ converges uniformly in $L^2(0,T;\mathcal{C}(\partial\Omega))$ and the normal derivative $\partial_{\nu}u$ of the solution $u$   of \eqref{eqmu}  can have the following spectral representation
    \begin{align}\label{bound-rep}
        \partial_{\nu}u(t,z) = \sigma(t)\partial_{\nu}G(z) - \sum_{n=1}^{+\infty}\left(\int_0^te^{-\lambda_n(t-s)}\sigma'(s)\frac{\scalarprod{\mu,\phi_n}}{\lambda_n}ds\right)\partial_{\nu}\phi_n(z),\quad (t,z)\in \Sigma.
    \end{align}
\end{prop}

\begin{proof}
    We remark that $G\in H^3(C_r)$ at the boundary so its normal derivative $\partial_{\nu}G$ is well defined and continuous on $\mathbb{S}^1$. We know from Theorem \ref{Fourierseriesofu} that the series 
$$\sum_{n\in\mathbb N}\left(\int_0^te^{-\lambda_n(t-s)}\sigma'(s)\frac{\scalarprod{\mu,\phi_n}}{\lambda_n}ds\right)\phi_n$$
converges in $L^2(0,T;H^{3-\delta}(\Omega))$, $\delta\in(0,1/2)$. By the continuity of the map $v\mapsto\partial_\nu v$ from $L^2(0,T;H^{3-\delta}(\Omega))$ to $L^2(0,T;H^{\frac 3 2-\delta}(\partial\Omega))$, we know that the series 
$$\partial_\nu\left(\sum_{n\in\mathbb N}\left(\int_0^te^{-\lambda_n(t-s)}\sigma'(s)\frac{\scalarprod{\mu,\phi_n}}{\lambda_n}ds\right)\phi_n\right)|_{\Sigma}=\sum_{n\in\mathbb N}\left(\int_0^te^{-\lambda_n(t-s)}\sigma'(s)\frac{\scalarprod{\mu,\phi_n}}{\lambda_n}ds\right)\partial_\nu\phi_n|_{\partial\Omega}$$
converges in $L^2(0,T;H^{\frac 3 2-\delta}(\partial\Omega))$ and, the Sobolev embedding theorem implies that this convergence occurs in $L^2(0,T;\mathcal C(\partial\Omega))$. Using these properties, we can easily deduce \eqref{bound-rep}.

\end{proof}

In addition to the above property of boundary representation, using the fact that $\sigma$ is constant on $(T_1,T)$, we can derive properties of analyticity in time of $\partial_\nu u$. Such property follows from the following result.

\begin{prop}\label{ana}
We can define on $P := \left\lbrace\ z\in\C\; ;\;Re(z)>T_1\right\rbrace$ the map $F$ by
\begin{align*}
        F(z) := \sum_{n = 1}^{+\infty}\left(\int_0^{T_1}e^{-\lambda_n(z-s)}\sigma'(s)\frac{\scalarprod{\mu,\phi_n}}{\lambda_n}ds\right)\partial_{\nu}\phi_n.
    \end{align*}
Moreover, $F$ is holomorphic on $P$ as a map taking values in $\mathcal C(\partial\Omega)$.
\end{prop}

\begin{proof}
 For $n\in\N$, we set $ E_n(z) := \int_0^{T_1}e^{-\lambda_n(z-s)}\sigma'(s)\frac{\scalarprod{\mu,\phi_n}}{\lambda_n}ds$, $
        F_n(z) := E_n(z)\partial_{\nu}\phi_n$,
    where $z\in\mathbb C$ satisfies $\Re(z)>T_1$. Thanks to the Cauchy-Schwarz inequality and the fact that 
    $-\lambda_n(\Re(z)-s)<-\lambda_n(T_1-s)$ for $s\in(0,T_1)$ then we have for $z\in P$,
    \begin{align*}
        \abs{E_n(z)}^2\leq \left(\int_0^{T_1}e^{-\lambda_n(\Re(z)-s)}\abs{\sigma'(s)}ds\right)^2\frac{\abs{\scalarprod{\mu_j,\phi_n}}^2}{\lambda_n^2}
        &\leq \frac{\abs{\scalarprod{\mu_j,\phi_n}}^2}{\lambda_n^2}\norm{\sigma'}^2_{L^2(0,T_1)}\int_0^{T_1}e^{-2\lambda_n(T_1-s)}ds\\
        &\leq \frac{\abs{\scalarprod{\mu_j,\phi_n}}^2}{2\lambda_n^3}\norm{\sigma'}^2_{L^2(0,T_1)}\leq C\norm{\sigma'}_{L^2(0,T_1)}^2.
    \end{align*}
    The integrand of $E_n$ is holomorphic with respect to $z$ in the half plane $P$ so $E_n$ is holomorphic in the same domain like $F_n$. And now we use the Weierstrass Theorem for holomorphic functions. So we prove that $\left(\sum\limits_{n=1}^NF_n(z)\right)_{N\geq 1}$ is uniformly a Cauchy sequence on every compact $K$ in the half plane $P$. Fix $K\subset P$ a compact then there exists $a,b\in\R$ such that $T_1< a\leq\Re(z)\leq b$ for $z\in K$. We also observe that for $1\leq N<M$ and $z\in K$,
    \begin{align*}
        \normm{\sum_{n=N+1}^MF_n(z)}_{\mathcal C(\partial\Omega)}=  \normm{\partial_{\nu}\left(\sum_{n=N+1}^ME_n(z)\phi_n\right)}_{\mathcal C(\partial\Omega)}\leq \left\lVert\sum_{n=N+1}^ME_n(z)\phi_n\right\rVert_{\mathcal{C}^1(\overline{\Omega})}.
    \end{align*}
     Hence by the embeddings $D(A^{1+\varepsilon})\hookrightarrow H^{2+2\varepsilon}(\Omega)\hookrightarrow\mathcal{C}^1(\overline{\Omega})$ for $\varepsilon\in(0,\frac{1}{2})$, in view of the Cauchy-Schwarz inequality we obtain for every $1\leq N< M$,
    \begin{align*}
        \sup_{z\in K}\left|\sum_{n=N+1}^MF_n(z)\right|^2
        &\leq C\sup_{z\in K}\left\lVert\sum_{n=N+1}^ME_n(z)\phi_n\right\rVert_{H^{2+2\varepsilon}(\Omega)}^2\\
        &\leq C\sup_{z\in K}\left\lVert\sum_{n=N+1}^ME_n(z)\phi_n\right\rVert_{D(A^{1+\varepsilon})}^2\\
        &\leq C\sup_{z\in K}\sum_{n=N+1}^M \lambda_n^{2+2\varepsilon}\abs{E_n(z)}^2\\
        &\leq C\norm{\sigma'}_{L^2(0,T)}^2\sum_{n=N+1}^M\lambda_n^{1+2\varepsilon}e^{-2\lambda_n(a-T_1)}\frac{\abs{\scalarprod{\mu_j,\phi_n}}^2}{2\lambda_n^2}.
    \end{align*}
    Recalling that the sequence $(\lambda_n^{1+2\varepsilon}e^{-2\lambda_n(a-T_1)})_{n\geq 1}$ is bounded since $a-T_1>0$ and $\sum\limits_{n=1}^{+\infty}\frac{\abs{\scalarprod{\mu_j,\phi_n}}^2}{\lambda_n^2}<+\infty$, we deduce that $\left(\sum\limits_{n=1}^NF_n(z)\right)_{N\geq 1}$ is uniformly a Cauchy sequence on $K$ as a map taking values in $\mathcal C(\partial\Omega)$. So the series $\sum\limits_{n=1}^{+\infty}F_n(z)$ converges uniformly on $K$ to the function $F$ for every compact $K\subset P$ as maps taking values in $\mathcal C(\partial\Omega)$. Therefore, by the Weierstrass Theorem, $F$ is holomorphic on the half plane $P$ as maps taking values in $\mathcal C(\partial\Omega)$.
\end{proof}

\section{Proof of Theorem \ref{inv_prob}}\label{sectionproof}
The goal of this section is to combine the results of all the preceding sections for completing the proof of the main result stated in Theorem \ref{inv_prob}. In contrast to the previous section, we denote by $\{\lambda_n:\ n\in\mathbb N\}$ the class of strictly increasing sequence of eigenvalues of $A$ and $d_n$ the algebraic multiplicity of $\lambda_n$. For each eigenvalue $\lambda_n$, we introduce an orthonormal basis $\{\phi_{n,k}\}_{k=1}^{d_n}$ of eigenspace of $A$ associated with $\lambda_n$.  Using the fact that $\Omega$ is the unit disc of $\R^2$, we consider explicit form of the orthonormal basis $\{\phi_{n,k}: n\in \mathbb N,\ k=1,\ldots,d_n\}$ of $L^2(\Omega)$ of eigenvectors of the operator $A$ by mean of Bessel functions and harmonic spheric. More precisely, following \cite[Theorem 2.1]{gong2026identification} and \cite[Section 2.2]{Li2020fractionalorder_sourceterm}, from now on, for $n\in\mathbb N$ and $k\in\{1,\ldots,d_n\}$, we fix $\phi_{n,k}$ explicitly expressed in polar coordinates by
\begin{align}\label{eigen}
   \phi_{n,k}(r,\theta) = \omega_n J_{m(n)}(\sqrt{\lambda_n} r)\exp(i (-1)^{k+1} m(n)\theta),\quad r\in(0,1),\ \theta\in(-\pi,\pi),\ k=1,\ldots,d_n
\end{align}
where $m:\mathbb N\to\mathbb N$, for all $\ell\in\mathbb N$, $J_\ell$ is the Bessel function of the first kind of order $\ell$ and $\omega_n =(\sqrt{\pi} J_{\abs{m(n)}+1}(\sqrt{\lambda_n}))^{-1}$. In this representation of the eigenfunctions, for every $n\in\mathbb N$, $m(n)$ is chosen uniquely   in such way that $\sqrt{\lambda_n}$ is a positive root of the Bessel function $J_{m(n)}$ (see e.g. \cite[Remark 2.1]{Li2020fractionalorder_sourceterm}).
As observed in \cite[Theorem 2.1]{gong2026identification} and \cite[Section 2.2]{Li2020fractionalorder_sourceterm}, 
we have $d_1=1$, $m(1)=0$ and, for $n\geq 2$, we get $d_n=2$ and $m(n)\geq1$ and one can check that $\phi_{n,2}=\overline{\phi_{n,1}}$.

Armed with the properties of the preceding sections and the above mentioned properties of the eigenpairs $(\lambda_n,\phi_{n,1})$, $(\lambda_n,\overline{\phi_{n,1}})$, $n\in\mathbb N$, with $\phi_{n,1}$ given by \eqref{eigen}, we are in position to complete the proof of Theorem \ref{inv_prob}.

\begin{proof}[\bf{Proof of Theorem \ref{inv_prob}}] We divide the proof of the theorem into four steps.\\
    \textit{Step 1 : Analytic extension}. 
    In view of Proposition \ref{ana}, for $j,\ell=1,2$, we can define on $(0,+\infty)$ the map $h_j^\ell$ by
    $$h_j^\ell(t) = \sigma(t)\partial_{\nu}G_j(x_\ell)-\sum_{n=1}^{+\infty}\sum_{k=1}^{d_n}\frac{\scalarprod{\mu_j,\phi_{n,k}}}{\lambda_n}\partial_{\nu}\phi_{n,k}(x_\ell)\int_0^te^{-\lambda_n(t-s)}\sigma'(s)ds,\quad t\in(0,+\infty),$$
    where $\sigma$ denotes the extension   to $(0,+\infty)$ of the map $\sigma$ satisfying $\sigma= \sigma(T)$ on $[T,+\infty)$. In this step, we will prove the following identity
  \begin{equation}\label{t1a}h_1^\ell(t)=h_2^\ell(t),\quad t\in(0,+\infty).\end{equation}
Using the fact that $\sigma=\sigma(T)$ on $(T_1,T)$, we get
    $$h_j^\ell(t) =\sigma(T)\partial_\nu G_j(x_\ell) -\sum_{n=1}^{+\infty}\sum_{k=1}^{d_n}\frac{\scalarprod{\mu_j,\phi_{n,k}}}{\lambda_n}\partial_{\nu}\phi_{n,k}(x_\ell)\int_0^{T_1}e^{-\lambda_n(t-s)}\sigma'(s)ds,\quad t\in(T_1,+\infty),$$
    and applying again Proposition \ref{ana}, we deduce that $h_j^\ell$ is analytic on $(T_1,+\infty)$. In addition, condition \eqref{Hypothesis_normalderivative} and Proposition \ref{bound} imply that
$h_1^\ell(t)=\partial_\nu u_1(t,x_\ell)=\partial_\nu u_2(t,x_\ell)=h_2^\ell(t)$,  $\ell=1,2,\ t\in(0,T)$.
Combining this with the fact that $h_j^\ell$, $j,\ell=1,2$, is analytic on $(T_1,+\infty)$ and applying unique continuation for analytic functions, we obtain \eqref{t1a}.\\

\textit{Step 2 : Laplace transform of }$h_j^\ell$ $j,\ell=1,2$. This step will be devoted to the proof of the following identity
 \begin{equation}\label{t1b}\hat{\sigma}(p)\left(\partial_{\nu}G(x_\ell) - \sum_{n=1}^{+\infty}p\frac{\sum\limits_{k=1}^{d_n} \scalarprod{\mu,\phi_{n,k}}\partial_{\nu}\phi_{n,k}(x_\ell)}{\lambda_m(p+\lambda_m)}\right) = 0,\quad p\in\C_+:=\{z\in\mathbb C:\ \textrm{Re}(z)>0\},\end{equation}
where $G=G_1-G_2$ and $\mu=\mu_1-\mu_2$.
For this purpose, we consider properties of Laplace transform of the map $h^\ell =h_1^\ell-h_2^\ell$ and $g^\ell=h^\ell-\sigma G$. Fix $\epsilon\in(0,\frac14)$. Using the continuous embeddings $X^{2+2\varepsilon}\hookrightarrow{C}^1(\overline{\Omega})$, applying Cauchy-Schwarz inequality, Fubini's theorem, Young's inequality for convolution products and condition \eqref{mureg}, for all $\tau>0$, we get
$$\begin{aligned}\int_0^{+\infty}e^{-\tau t}|g^\ell(t)|dt&\leq \frac{1}{\sqrt{\tau}}\left(\int_0^{+\infty}e^{-\tau t}|g^\ell(t)|^2dt\right)^{\frac12}\\
&\leq \frac{1}{\sqrt{\tau}} \left(\int_0^{+\infty}e^{-\tau t}\norm{\sum_{n=1}^{+\infty}\sum_{k=1}^{d_n}\frac{\scalarprod{\mu,\phi_{n,k}}}{\lambda_n}\phi_{n,k}\int_0^te^{-\lambda_n(t-s)}\sigma'(s)ds}_{C^1(\overline{\Omega})}^2\right)^{\frac12}\\
&\leq C\left(\int_0^{+\infty}e^{-\tau t}\norm{\sum_{n=1}^{+\infty}\sum_{k=1}^{d_n}\frac{\scalarprod{\mu,\phi_{n,k}}}{\lambda_n}\phi_{n,k}\int_0^te^{-\lambda_n(t-s)}\sigma'(s)ds}_{X^{2+2\epsilon}}^2\right)^{\frac12}\\
&\leq C\left(\sum_{n=1}^{+\infty}\sum_{k=1}^{d_n}|\scalarprod{\mu,\phi_{n,k}}|^2\lambda_n^{2\epsilon}\underbrace{\left(\int_0^{+\infty}e^{-\tau t}\int_0^te^{-\lambda_n(t-s)}|\sigma'(s)|dsdt\right)^2}_{\leq \norm{\sigma'}_{L^1(0,T)}^2(\tau+\lambda_n)^{-2}}\right)^{\frac12}\\
&\leq C\left(\sum_{n=1}^{+\infty}\sum_{k=1}^{d_n}|\lambda_n^{2(\epsilon-1)}\scalarprod{\mu,\phi_{n,k}}|^2\right)^{\frac12}=C\norm{\mu}_{X^{-2(1-\epsilon)}}<\infty,
\end{aligned}$$
where $C>0$ is a constant changing from line to line. Therefore, for all $p\in\mathbb C_+$, $\hat{h}^\ell(p)$ is well-defined and, repeating the above argumentation combined with Lebesgue dominate convergence theorem, one can easily check that 
$$\begin{aligned}\hat{h}^\ell(p)&=\hat{\sigma}(p)\partial_{\nu}G(x_\ell)-\sum_{n=1}^{+\infty}\sum_{k=1}^{d_n}\frac{\scalarprod{\mu_j,\phi_{n,k}}}{\lambda_n}\partial_{\nu}\phi_{n,k}(x_\ell)\left(\int_0^{+\infty}e^{-pt}\int_0^te^{-\lambda_n(t-s)}\sigma'(s)dsdt\right)\\
&=\hat{\sigma}(p)\partial_{\nu}G(x_\ell)-\underbrace{\widehat{\sigma'}(p)}_{p\widehat{\sigma}(p)}\sum_{n=1}^{+\infty}\sum_{k=1}^{d_n}\frac{\scalarprod{\mu_j,\phi_{n,k}}}{\lambda_n(\lambda_n+p)}\partial_{\nu}\phi_{n,k}(x_\ell).\end{aligned}$$
On the other hand, condition \eqref{t1a} implies that $h^\ell=h_1^\ell-h_2^\ell\equiv0$ and, using the above identity we easily deduce \eqref{t1b}.\\

\textit{Step 3 : Fundamental identity. } In this step we will show the following identity
\begin{equation}\label{t1c}\sum\limits_{k=1}^{d_n} \scalarprod{\mu,\phi_{n,k}}\partial_{\nu}\phi_{n,k}(x_\ell)=0,\quad n\in\mathbb N,\ \ell=1,2.\end{equation}
Since $\sigma$ is not uniformly vanishing and $\hat{\sigma}$ is holomorphic in $\mathbb C_+$, by application of the isolated zero theorem, there exists $r_1,r_2\in(0,+\infty)$, $r_1<r_2$, such that $\hat{\sigma}(p)\neq0$, $p\in(r_1,r_2)$. Combining this with \eqref{t1b}, for all $p\in (r_1,r_2)$,  we obtain
\begin{equation}\label{t1d}\partial_{\nu}G(x_\ell) - \sum_{n=1}^{+\infty}p\frac{\sum\limits_{k=1}^{d_n} \scalarprod{\mu,\phi_{n,k}}\partial_{\nu}\phi_{n,k}(x_\ell)}{\lambda_m(p+\lambda_m)}=0,\quad \ell=1,2.\end{equation}
Meanwhile, using \eqref{mureg}, one can easily check that the map 
$$p\mapsto\sum_{n=1}^{+\infty}p\frac{\sum\limits_{k=1}^{d_n} \scalarprod{\mu,\phi_{n,k}}\partial_{\nu}\phi_{n,k}(x_\ell)}{\lambda_m(p+\lambda_m)} $$
is holomorphic in $\mathcal O=\mathbb C\setminus\{-\lambda_n:\ n\in\mathbb N\}$. Then, by isolated zero theorem, we deduce that \eqref{t1d} holds true for $p\in\mathcal O$. Finally, fixing $n\in\mathbb N$, multiplying \eqref{t1d} by $p+\lambda_n$ and sending $p\to-\lambda_n$, we obtain \eqref{t1c}.

\textit{Step 4 : Completion of the proof. }
   In this step we complete the proof of the theorem. For this purpose, we will discuss the consequence of \eqref{t1c} depending on the value of the index $n$. Observe that the argument of this step fundamentally rely on the representation \eqref{eigen}
    of the eigenfunctions $\{\phi_{n,k}:\ n\in\mathbb N,\ k=1,\ldots,d_n\}$.
    As observed at the beginning of this section, we have $d_1=1$, $m(1)=0$ and, for $n\geq 2$, we get $d_n=2$ and $m(n)\geq1$. Using classical properties of Bessel functions (see e.g. \cite[Lemma 2.6]{Li2020fractionalorder_sourceterm}), for $\ell=1,2$, we find
    \begin{equation}\label{t1aa}
        \partial_{\nu}\phi_{n,1}(x_\ell)=\sqrt{\lambda_n}\omega_nJ_{m(n)}'(\sqrt{\lambda_n})e^{im(n)\theta}=-\sqrt{\lambda_n}\omega_nJ_{m(n)+1}(\sqrt{\lambda_n})e^{im(n)\theta_\ell}=-\sqrt{\frac{\lambda_n}{\pi}}e^{im(n)\theta_\ell}\neq0.\end{equation}
For $n=1$, since $d_1=1$, \eqref{t1c} can be rewritten as $\scalarprod{\mu,\phi_{1,1}}\partial_{\nu}\phi_{1,1}(x_\ell) = 0$ and \eqref{t1aa} implies that $\scalarprod{\mu,\phi_{1,1}} = 0$.

Now let us choose $n\geq2$ and recall that $d_n=2$ and $\phi_{n,2}=\overline{\phi_{n,1}}$. Then, \eqref{t1c} and \eqref{t1aa} imply that, for all $\ell=1,2$, we get
        \begin{align*}
            \scalarprod{\mu,\phi_{n,1}}\partial_{\nu}\phi_{n,1}(x_\ell) + \scalarprod{\mu,\overline{\phi_{n,1}}}\partial_{\nu}\overline{\phi_{n,1}}(x_\ell) = 0\iff -\sqrt{\frac{\lambda_n}{\pi}}\left(\scalarprod{\mu,\phi_{n,1}}e^{im(n)\theta_\ell}+\scalarprod{\mu,\overline{\phi_{n,1}}}e^{-im(n)\theta_\ell}\right) = 0.
        \end{align*} These equations can be equivalently reformulated as the following system:
        \begin{align*}
            \left(
            \begin{array}{cc}
                e^{im(n)\theta_1} & e^{-im(n)\theta_1}\\
                e^{im(n)\theta_2} & e^{-im(n)\theta_2}
            \end{array}
            \right)\left(
            \begin{array}{c}
                \scalarprod{\mu,\phi_{n,1}}\\
                \scalarprod{\mu,\overline{\phi_{n,1}}}
            \end{array}
            \right) = \left(
            \begin{array}{c}
                0\\
                0
            \end{array}
            \right).
        \end{align*}
        Recalling that $\theta_1 - \theta_2\notin\pi\Q$ and $m(n)\neq 0$, we obtain
        \begin{align*}
            e^{im(n)(\theta_1-\theta_2)} - e^{-im(n)(\theta_1-\theta_2)} = 2i\sin(m(n)(\theta_1-\theta_2))\neq 0.
        \end{align*} It follows that $\scalarprod{\mu,\phi_{n,1}} = \scalarprod{\mu,\overline{\phi_{n,1}}} = 0$.
        
From the above discussion, we deduce that
$\scalarprod{\mu,\phi_{n,k}}=0$, $n\in\mathbb N$, $k=1,\ldots,d_n$.
Thus, applying Lemma \ref{nullmeasure} in the Appendix, we deduce that $\mu$ is a null measure, or equivalently that $\mu_1=\mu_2$. This completes the proof of the theorem.
\end{proof}\bigskip

\section{Numerical results and discussions}\label{sec:num}
In this section, we present  the feasibility of recovering the unknown source $\mu(x)$  from two-point boundary measurements. Due to the sparsity of the measurements, reconstructing a general Radon measure $\mu(x)$ is numerically challenging. Thus, we focus on two special cases: 
\begin{enumerate}[label=(\roman*)]
  \item $\mu$ is supported on a finite set of points;
  \item $\mu$ is supported on a curve.
\end{enumerate}
The reconstruction algorithms and corresponding numerical experiments are detailed in the subsequent sections. Throughout, we fix terminal time $T=1$. The direct problems are solved by  the Galerkin finite element method (with conforming piecewise linear elements) in space and backward Euler scheme in time \cite{Thomee:2006}. We  employ a finer space-time mesh (with a mesh size $  \frac{1}{50}$ and time steps $\frac{1}{600}$) to generate the exact flux $\partial_{\nu}u$, and a coarse mesh (with a mesh size $\frac{1}{25}$ and time steps $\frac{1}{300}$) for the reconstruction.  We generate the noisy measurement $z^\delta=(z^\delta_1,z^\delta_2)\in (L^2(0,T))^2$ by
\begin{equation}\label{eqn:noise}
    z^\delta_\ell(t)=\partial_{\nu}u(t,x_\ell)(1+\delta \xi),\quad \ell=1,2,
\end{equation}
where $\delta>0$ denotes the relative noise level, $\xi$ follows the standard Gaussian noise and $\partial_{\nu}u$. 

\subsection{Reconstruction of point sources}\label{subsec:num_pt}
In this part, we consider the reconstruction of point sources. In particular, we assume that $\mu$ is a sum of Dirac measures:
\begin{equation*}
\mu(x)=\sum_{i=1}^N \delta_{p_i}(x),\qquad p_i\in\omega\subset\subset\Omega.
\end{equation*}
Therefore, recovering the measure $\mu$ is equivalent to identifying the locations of $p_i$. We employ the regularized Gauss-Newton method \cite[Chapter 10]{Nocedal:2006}. Specifically, we define a nonlinear operator $F:p=(p_1,\dots,p_N) \in \mathbb{R}^{2N}\to (\partial_\nu u(p)(t,x_1),\partial_\nu u(p)(t,x_2))\in (L^2(0,T))^2 $, where $u(p)$ solves problem \eqref{eqmu} with the parameter $p$. For an initial guess $p^0$, we implement the following iteration: 
\begin{equation*}
    p^{k+1}=\arg\min J_k(p),
\end{equation*}
with the functional $J_k(p)$ at the $k$th iteration (based on $p^k$) given by
\begin{equation*}
    J_k(p)=\frac{1}{2}\|F(p^k)-z^\delta+\partial_p F(p^k)(p-p^k) \|_{L^2(0,T)}^2+\frac{\lambda^k}{2}|p-p^k|^2,
\end{equation*}
where $\lambda^k$ is regularization parameter, chosen according to the
specific problem,  and $|\cdot|$ denotes the Euclidean norm.  The Jacobian $\partial_p F(p^k)$ is approximated by the central finite difference scheme $\partial_p F(p)\approx (2 \epsilon)^{-1}(F(p+\epsilon)-F(p-\epsilon))$, where $\epsilon$ is a small number, fixed at $\epsilon=1\times10^{-3}$ in all experiments. Since $J_k$ is quadratic, the minimizer satisfies the normal equation 
\begin{equation*}
    (J_p^* J_p+\lambda^k I)\mathrm{d}p^k=J_p^*(z^\delta-F(p^k)),
\end{equation*}
where $J_p=\partial_p F(p^k)$, $\mathrm{d}p^k=p^{k+1}-p^k$ and $J_p^*$ denotes the adjoint operator. 

Now we present numerical results for the point sources identification. The
accuracy of a reconstruction is measured by the relative $\ell^2(\mathbb{R}^{2N})$ error: $e_\mu(p^k)=|p^k-p^\dagger|/|p^\dagger|$, where $p^\dagger $ refers to the exact location of points.  The residual $r_\mu$  of the recovered source is computed as $r_\mu(p^k)=\|F(p^k)-z^\delta\|_{L^2(0,T)}$.

\begin{example}\label{ex:pt}
    Let $\sigma(t)= \sin(\pi t/0.8)^2  \chi_{[0,0.8]}(t)$, where $\chi$ is the characteristic function. The boundary observation points are $x_\ell=e^{i\theta_\ell}$, with $\theta_1=0$, $\theta_2=\sqrt{2}/2$. 
    \begin{enumerate} 
      \item[(i)] $p_1=(0.5,0.5)$;
      \item[(ii)] $p_1=(0.5,0.5)$, $p_2=(-0.5,-0.5)$; 
      \item[(iii)] $p_1=(0.5,0.5)$, $p_2=(-0.5,0.5)$, $p_3=(-0.5,-0.5)$;
      \item[(iv)] $p_1=(0.5,0.5)$, $p_2=(-0.5,0.5)$, $p_3=(-0.4,0.4)$, $p_4=(0.4, -0.4)$.
    \end{enumerate}
\end{example}
In Figure~\ref{Fig:ex_pt_loss}, we show the convergence behaviour  of the Gauss-Newton method. Throughout all reconstructions, the regularization parameter $\lambda$ is kept constant within each example: $\lambda = 1e$-4 for Example~\ref{ex:pt}(i) and (ii), $\lambda = 8e$-2 for Example~\ref{ex:pt}(iii), and $\lambda = 1.5e$-1 for Example~\ref{ex:pt}(iv).  The initial guess $p^0$ is chosen near the true parameter vector $p^\dagger$, with $|p^0 - p^\dagger|/|p^\dagger| \approx 0.1$. For Example~\ref{ex:pt}(i) and (ii), both residual and relative error decay rapidly and then stabilise. The inverse problem becomes increasingly ill-posed as the number of unknown point sources grows. In Example~\ref{ex:pt}(iii) and (iv), the relative error converges slowly or even increase when noise level is large. In Examples~\ref{ex:pt}(iii) and (iv), the relative error converges slowly, or even grows when the noise level is large.  This latter behaviour indicates that a stronger regularization term or an early-stopping strategy should be employed. Figure~\ref{Fig:ex_pt_recon} presents the recovered locations of $p$; detailed quantitative results are also provided in Table~\ref{tab:ex_pt}. The algorithm successfully recovers the locations of $p$ even in the presence of a $10\%$ noise level, demonstrating the robustness of the proposed method.

\begin{table}[hbt!]
  \centering
  \begin{threeparttable}
    \caption{The relative errors for Example~\ref{ex:pt} at different noise levels.\label{tab:ex_pt}}
    \begin{tabular}{c|cccc}
        \toprule
           $e_\mu$ &  Example~\ref{ex:pt}(i) &  Example~\ref{ex:pt}(ii) &  Example~\ref{ex:pt}(iii) &  Example~\ref{ex:pt}(iv) \\
        \midrule
             $\delta=1\%$ & 1.97e-3 & 1.35e-2 & 1.13e-2 & 1.82e-2\\
             $\delta=5\%$ & 6.06e-3 & 1.74e-2 & 1.93e-2 & 2.51e-2 \\
             $\delta=10\%$ & 1.54e-2 & 3.00e-2 & 2.53e-2 & 3.49e-2\\
        \bottomrule
    \end{tabular}
      \end{threeparttable}
\end{table}

\begin{figure}[htbp]
	\centering
	\begin{tabular}{cccc}
		\includegraphics[width=0.22\textwidth]{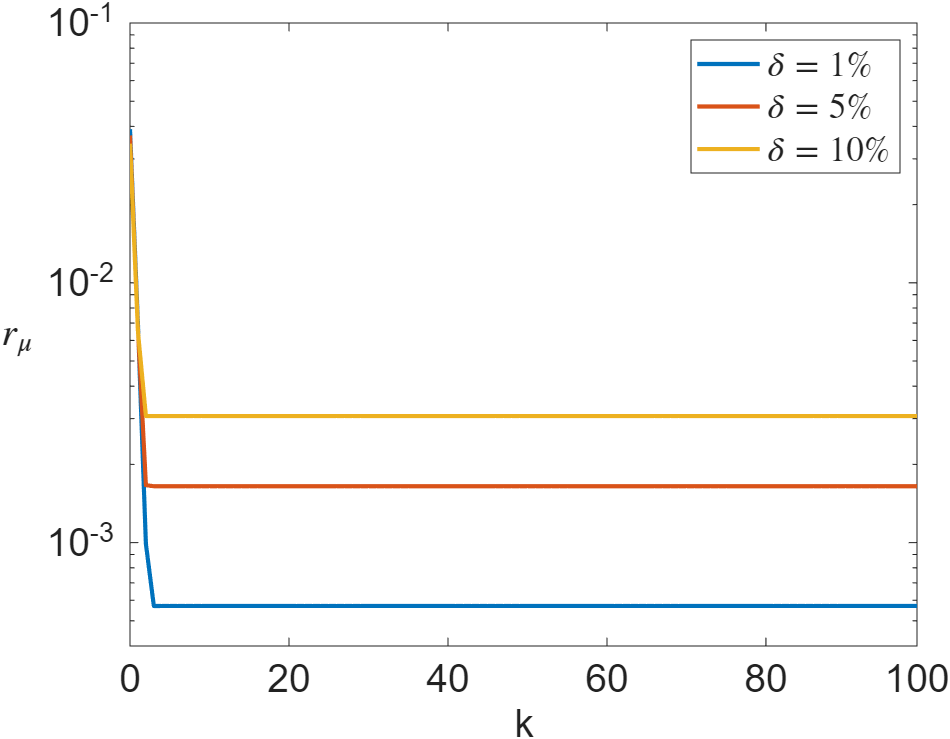}&
		\includegraphics[width=0.22\textwidth]{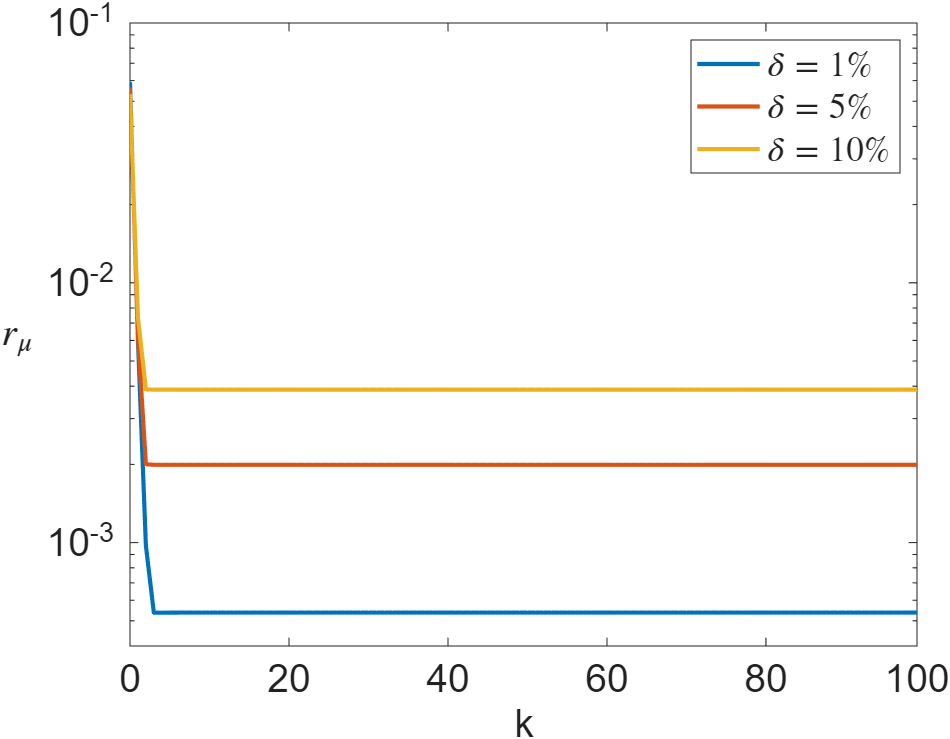}&
		\includegraphics[width=0.22\textwidth]{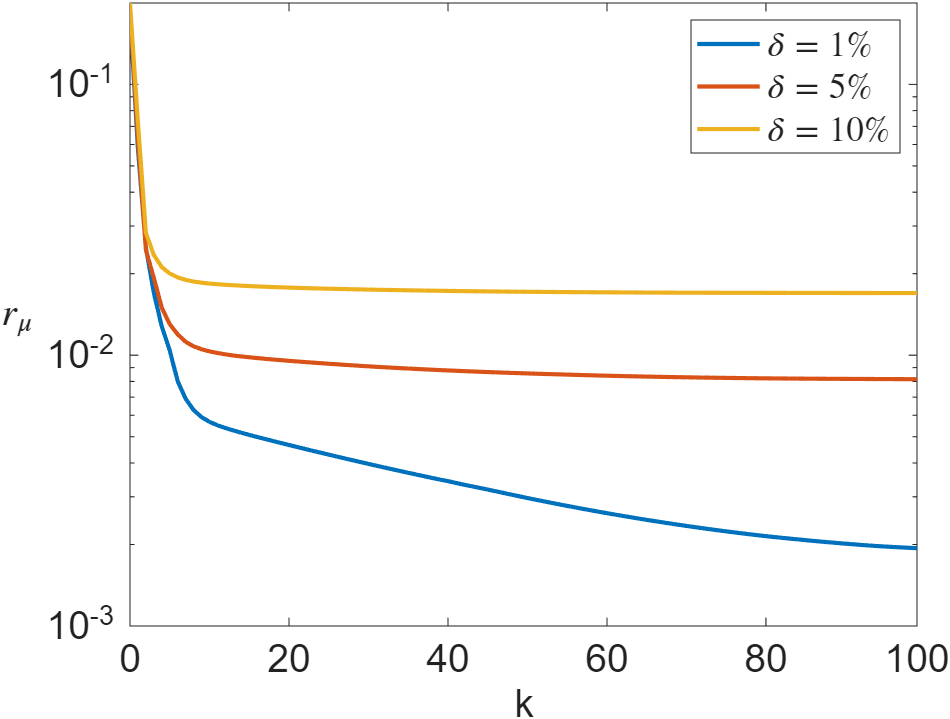}&
		\includegraphics[width=0.22\textwidth]{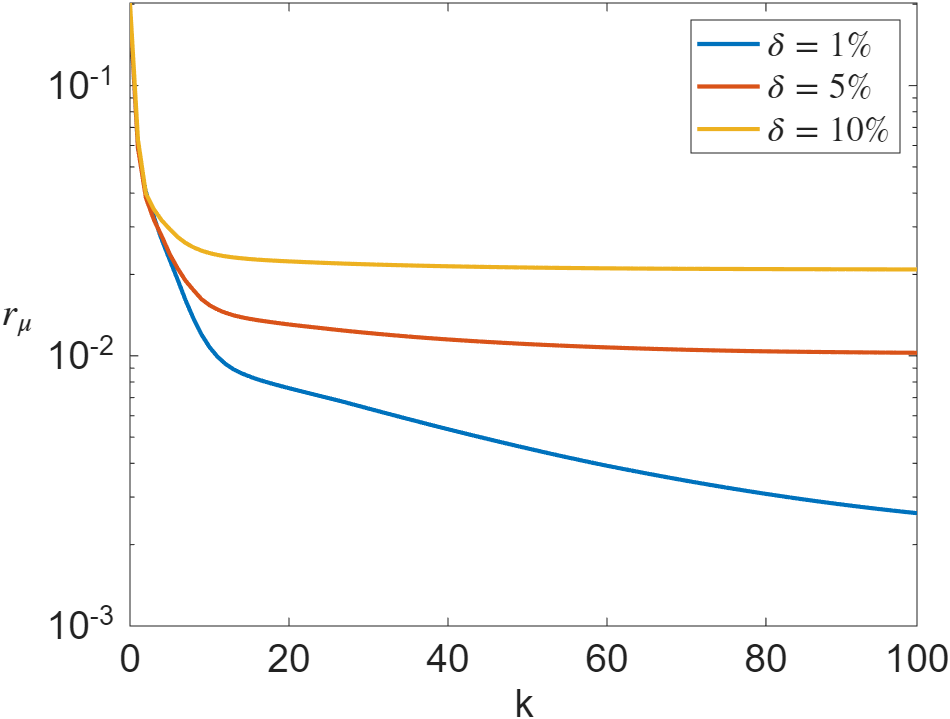}\\
        \includegraphics[width=0.22\textwidth]{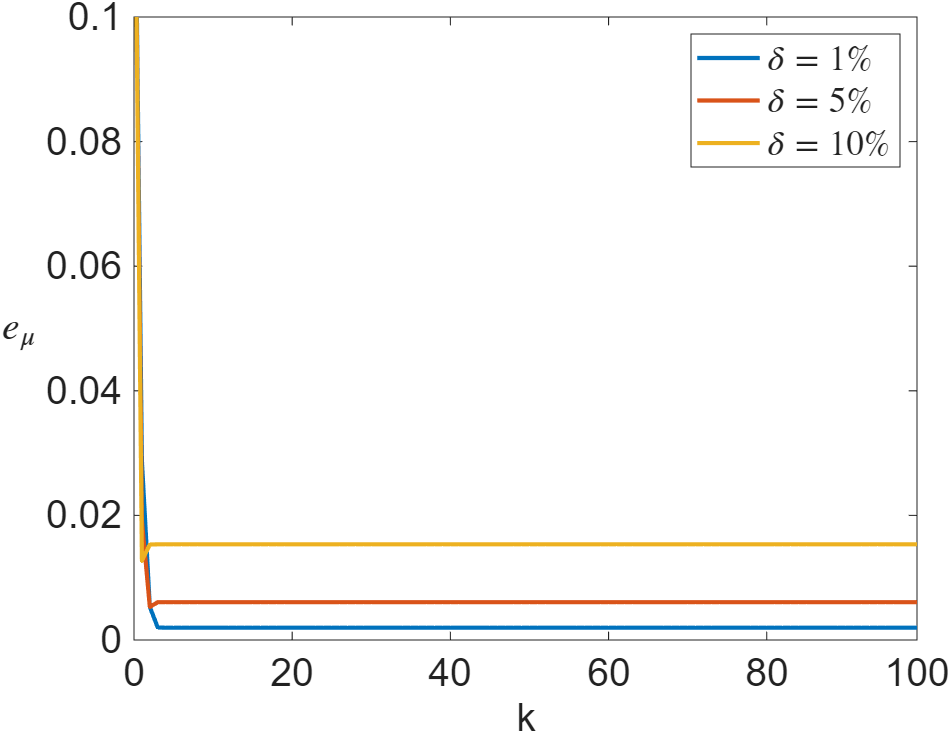}&
		\includegraphics[width=0.22\textwidth]{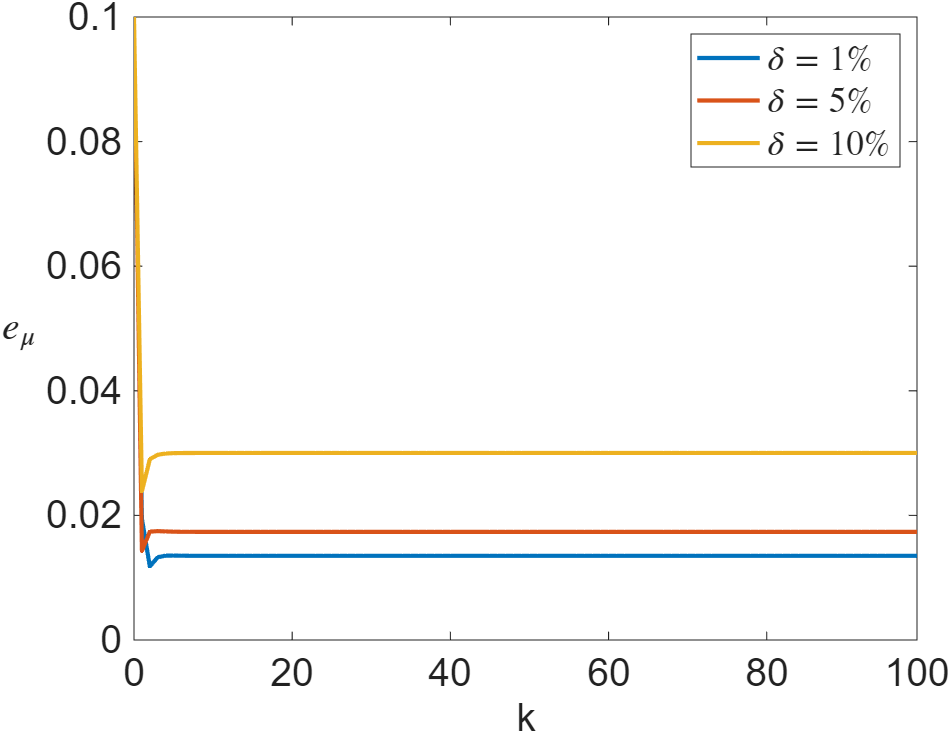}&
		\includegraphics[width=0.22\textwidth]{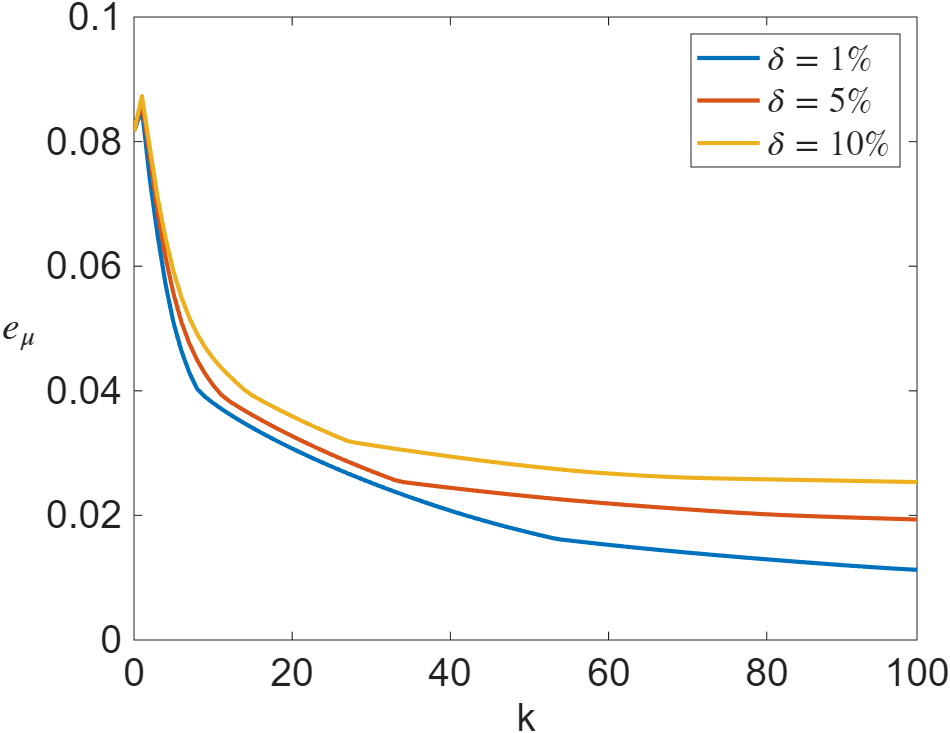}&
		\includegraphics[width=0.22\textwidth]{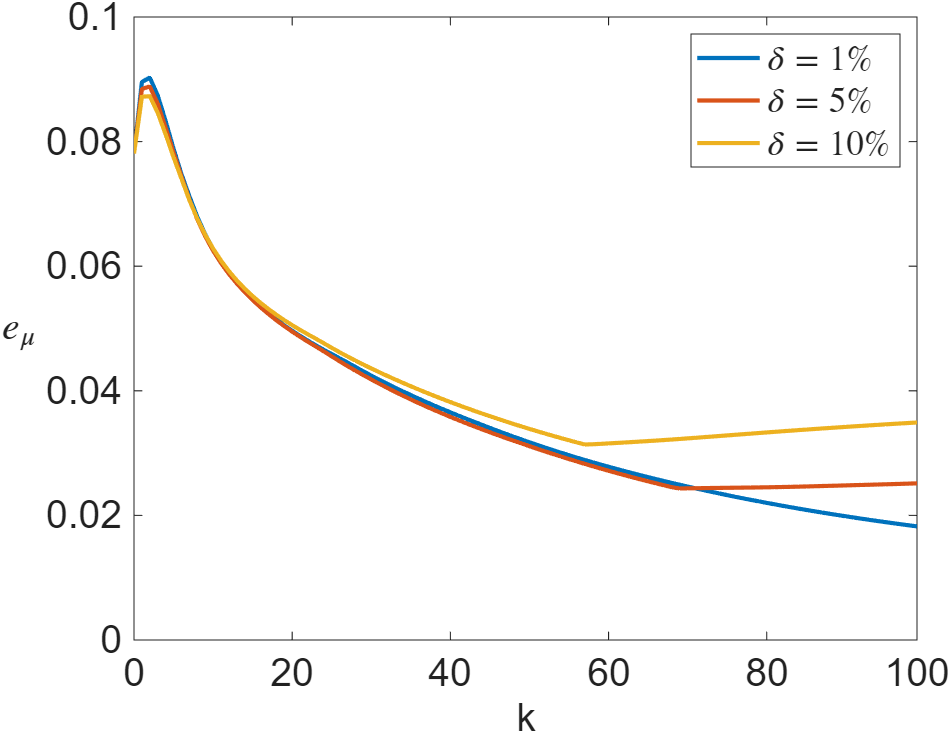}\\
		(a) Example~\ref{ex:pt}(i) & (b) Example~\ref{ex:pt}(ii)  & (c) Example~\ref{ex:pt}(iii) & (d) Example~\ref{ex:pt}(iv) 
	\end{tabular}
	\caption{Evolution of residuals $r_\mu$ and relative errors $e_\mu$ with respect to iteration number $k$ in Example~\ref{ex:pt}.}
	\label{Fig:ex_pt_loss}
\end{figure}

\begin{figure}[htbp]
	\centering
	\begin{tabular}{cccc}
		\includegraphics[width=0.22\textwidth]{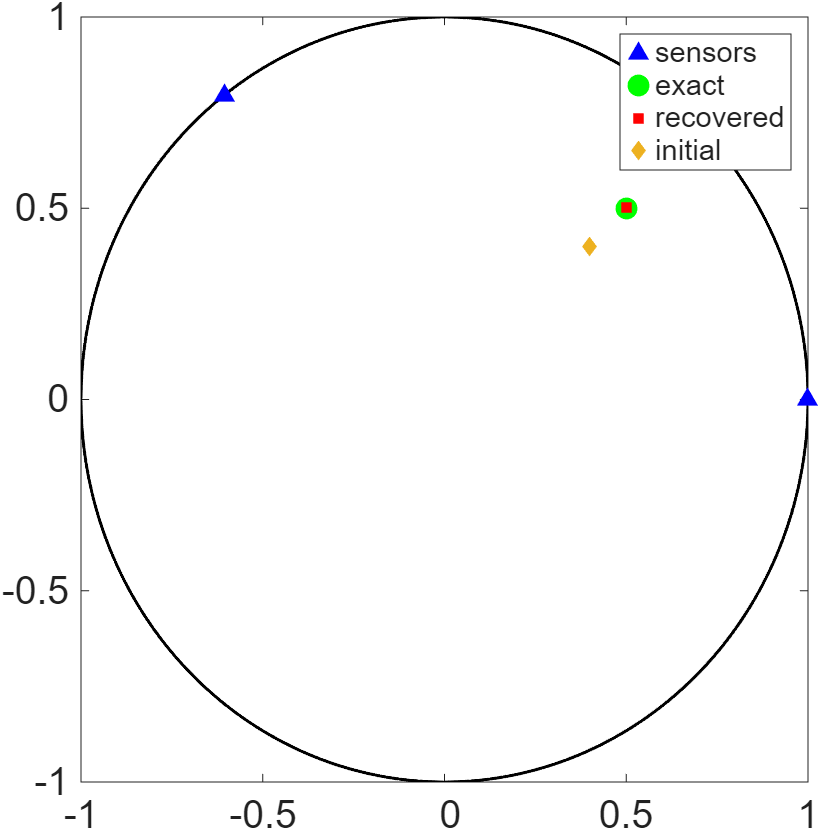}&
		\includegraphics[width=0.22\textwidth]{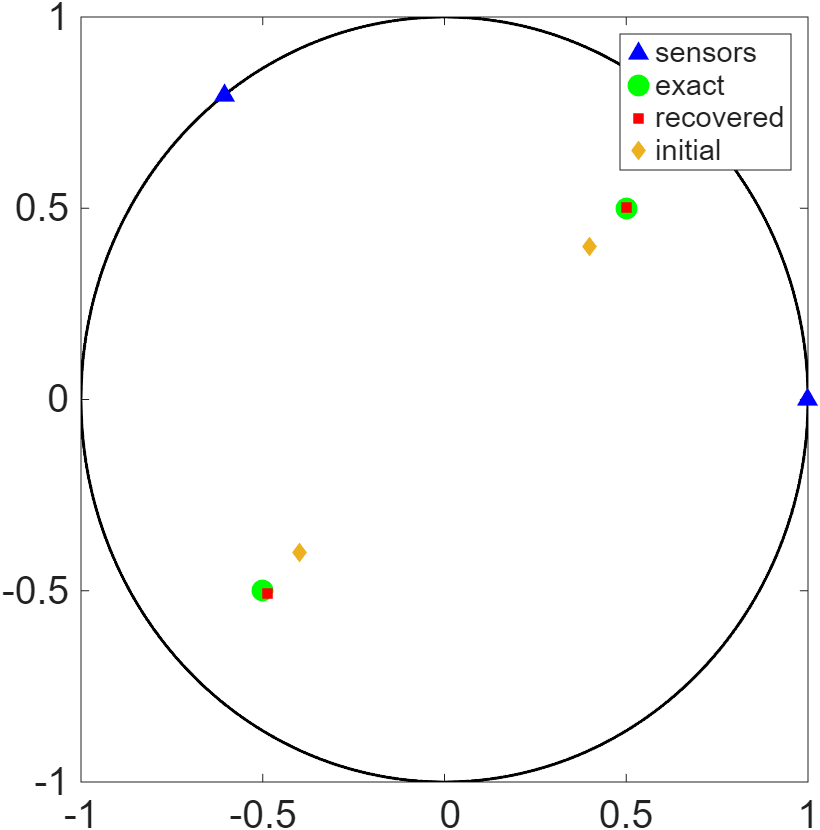}&
		\includegraphics[width=0.22\textwidth]{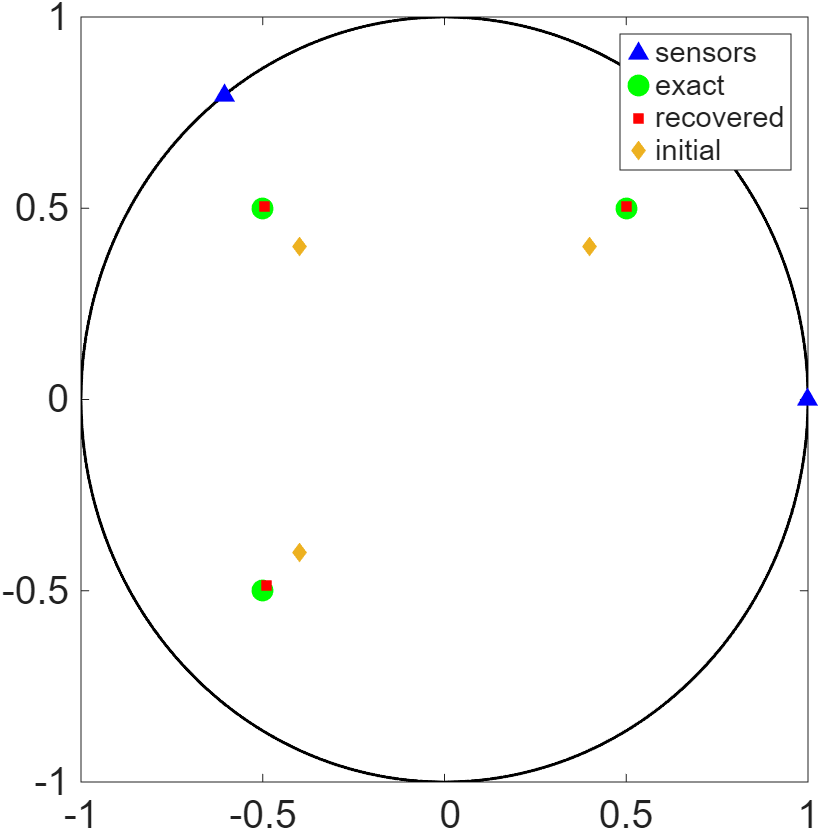}&
		\includegraphics[width=0.22\textwidth]{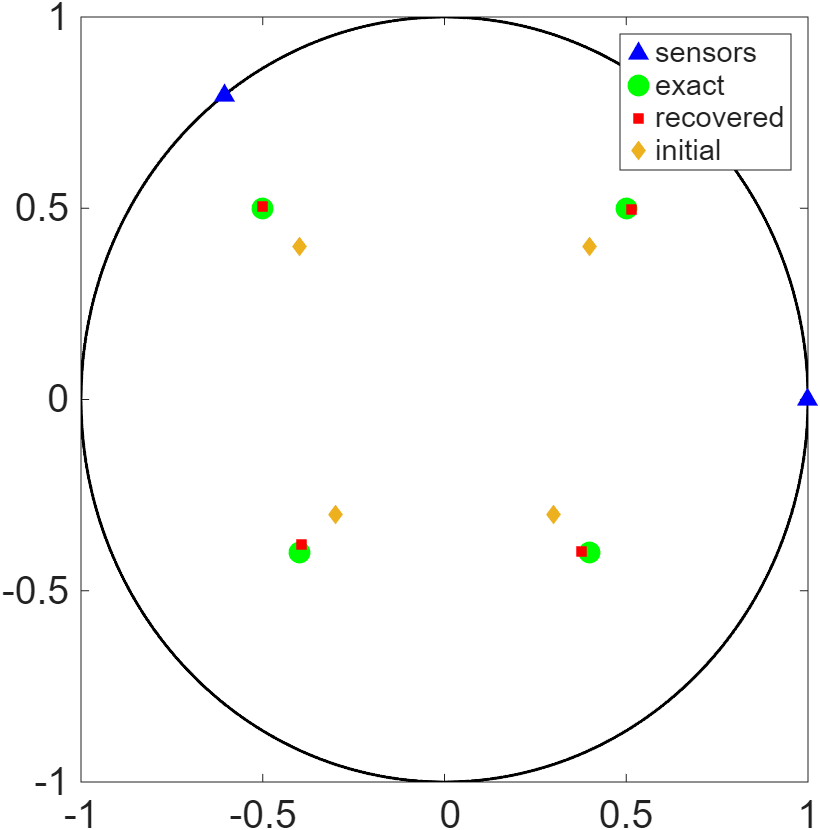}\\
		\includegraphics[width=0.22\textwidth]{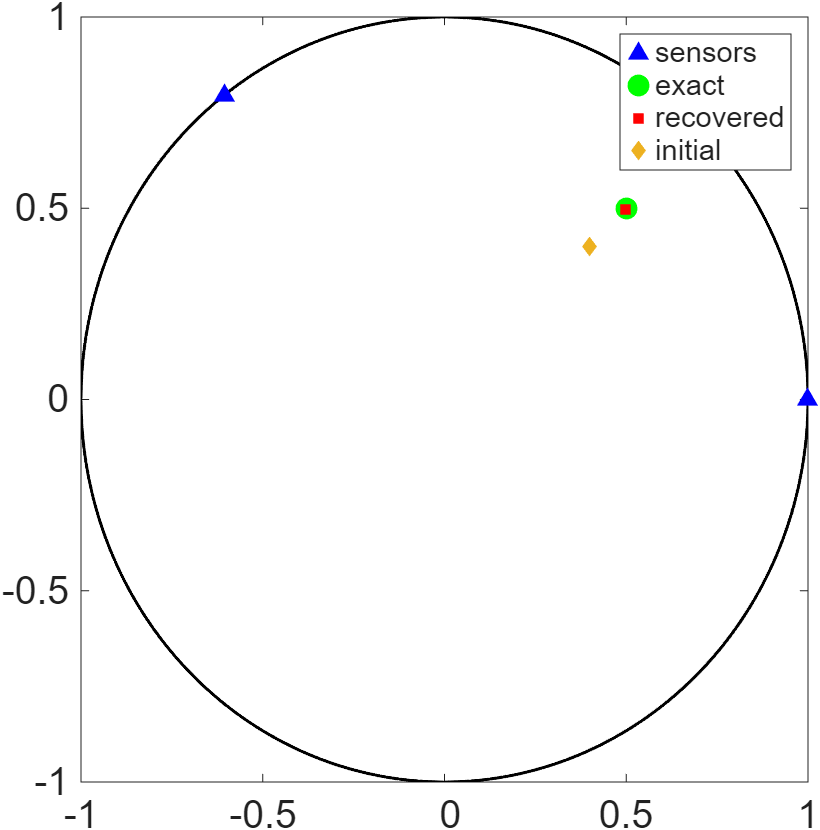}&
		\includegraphics[width=0.22\textwidth]{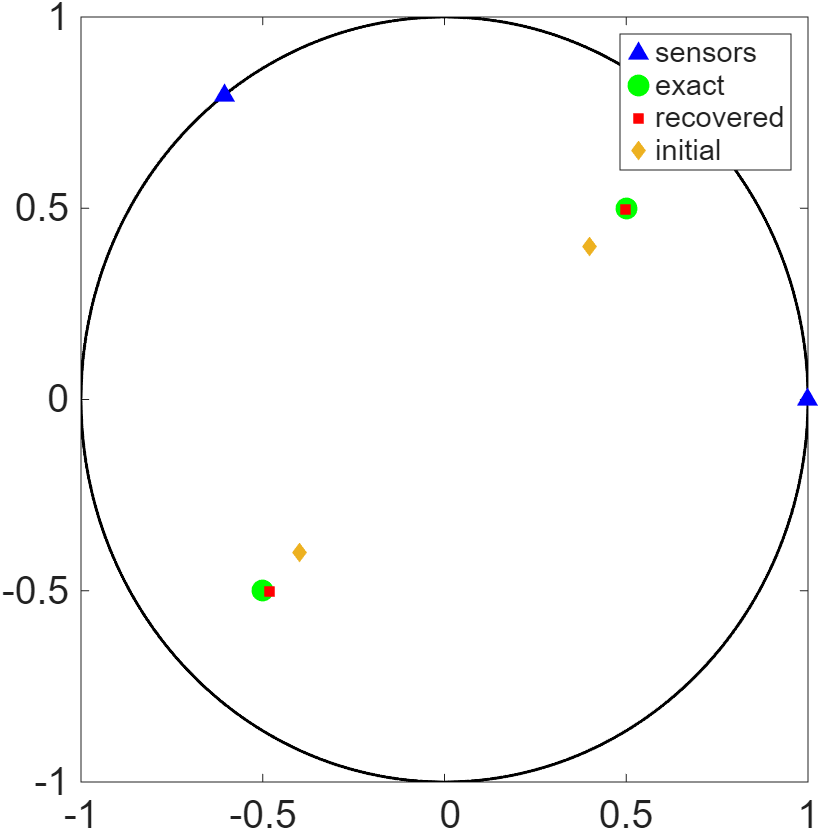}&
		\includegraphics[width=0.22\textwidth]{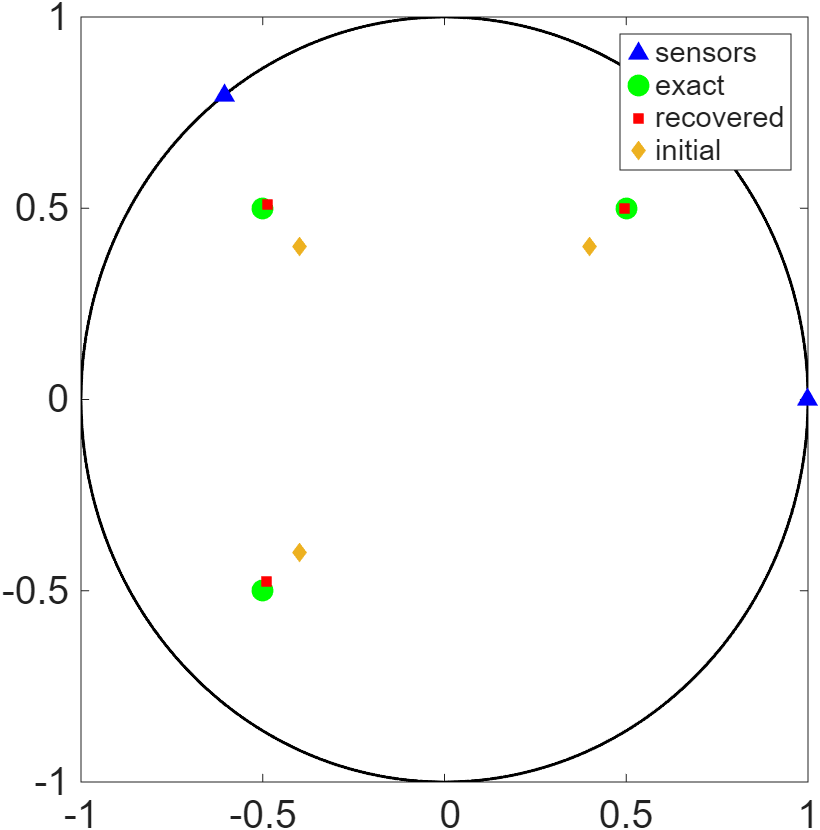}&
		\includegraphics[width=0.22\textwidth]{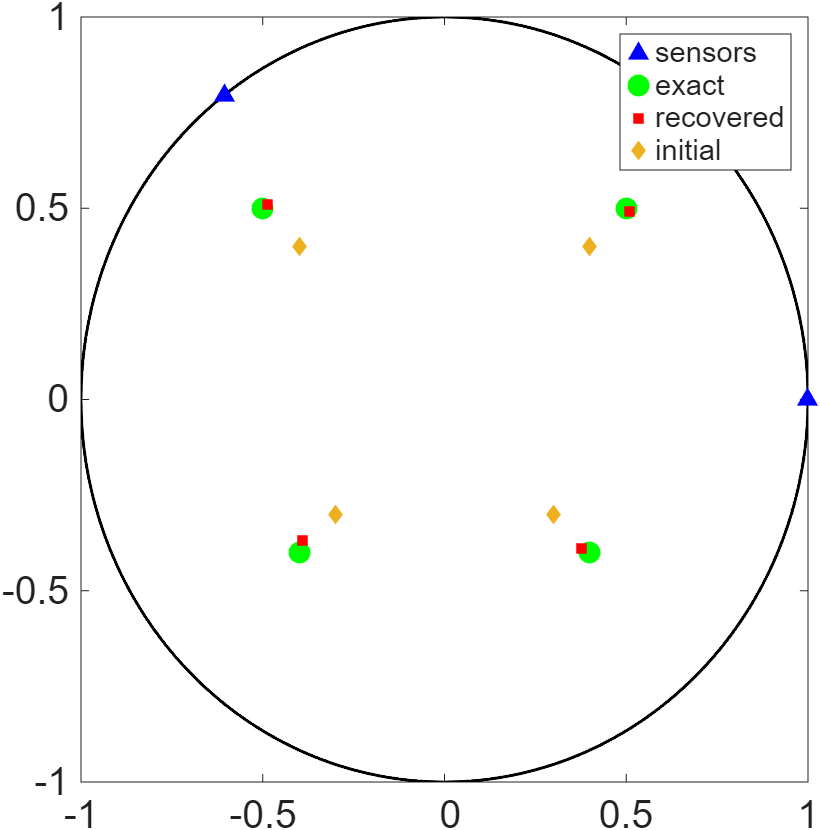}\\
        \includegraphics[width=0.22\textwidth]{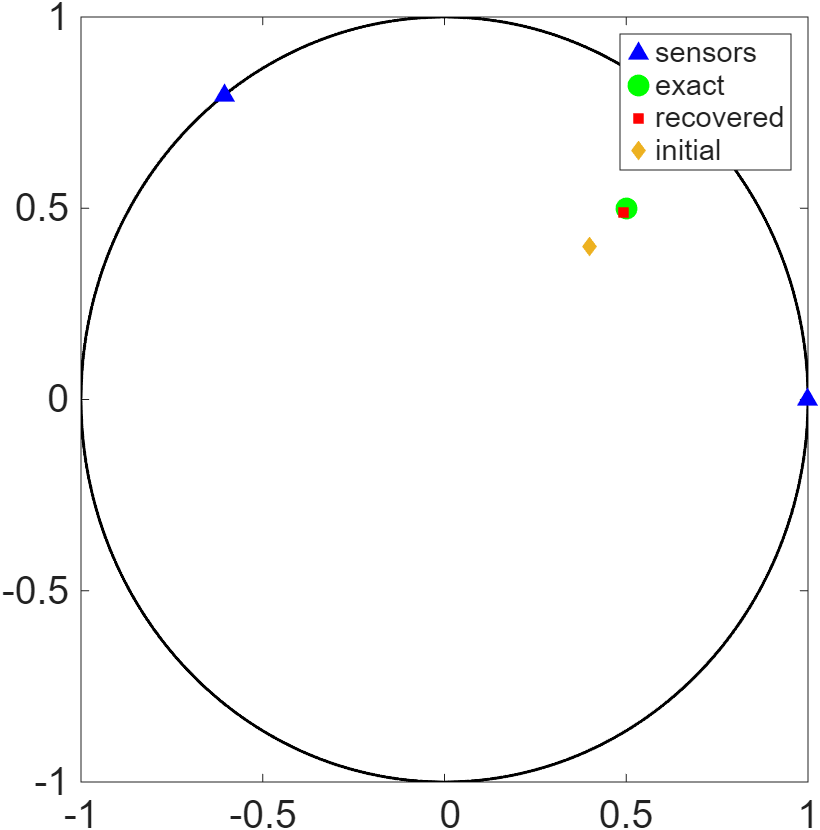}&
		\includegraphics[width=0.22\textwidth]{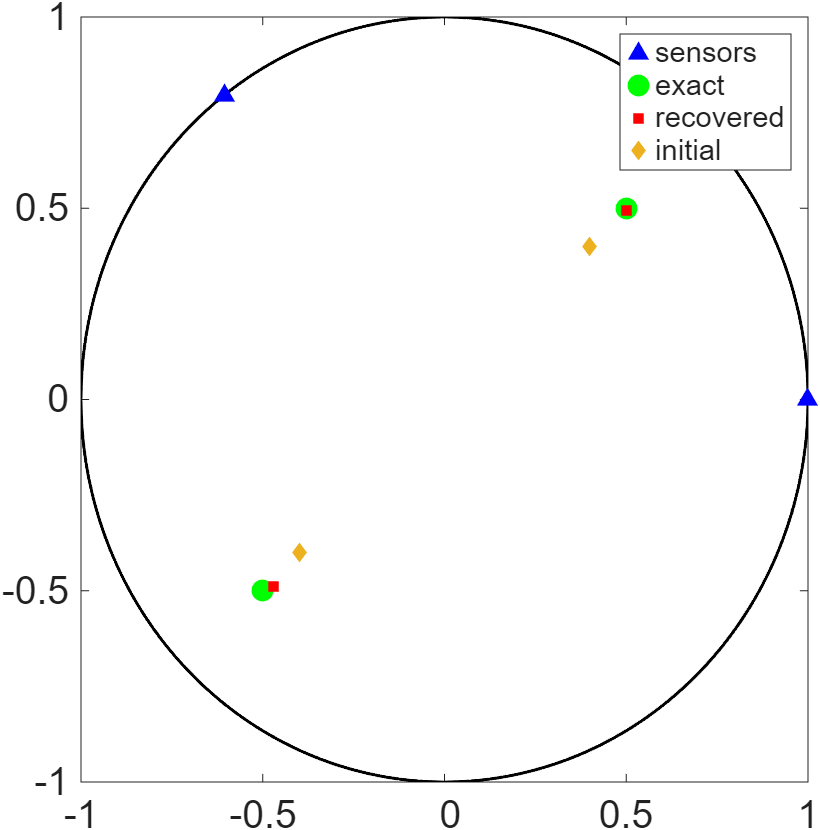}&
		\includegraphics[width=0.22\textwidth]{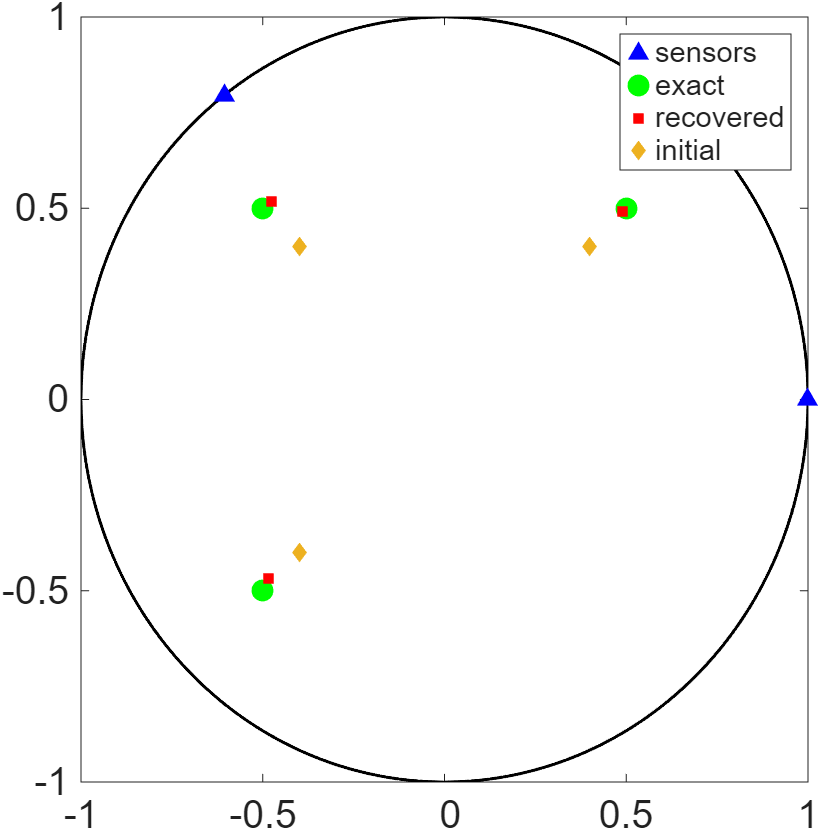}&
		\includegraphics[width=0.22\textwidth]{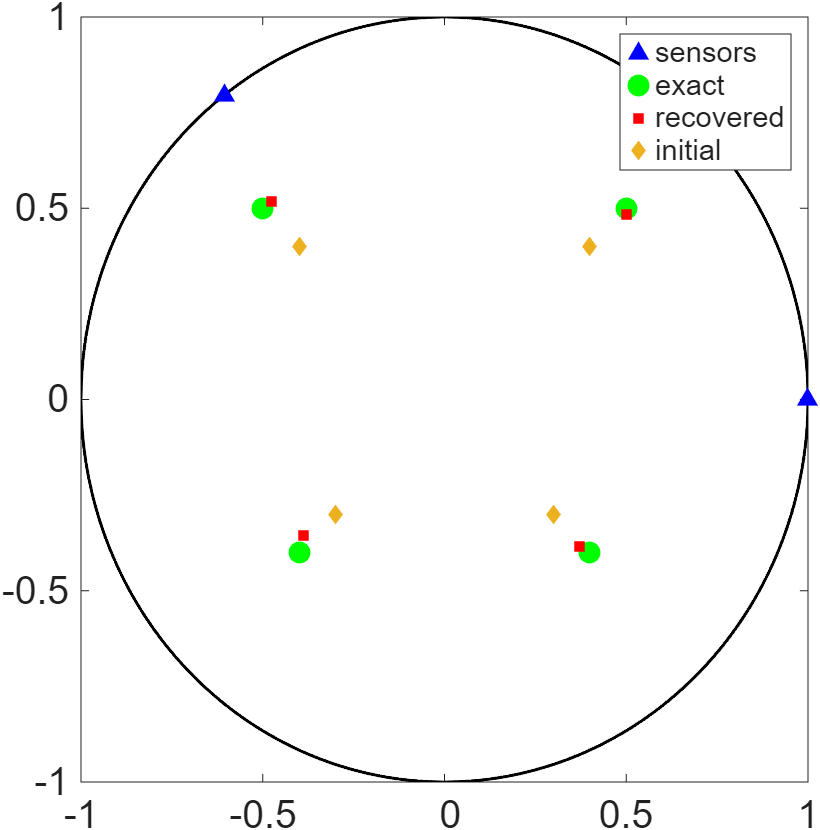}\\
		(a) Example~\ref{ex:pt}(i) & (b) Example~\ref{ex:pt}(ii)  & (c) Example~\ref{ex:pt}(iii) & (d) Example~\ref{ex:pt}(iv) 
	\end{tabular}
	\caption{The reconstruction  for Example~\ref{ex:pt} at different noise levels. From top to bottom: $\delta=1\%$, $\delta=5\%$, $\delta=10\%$.}
	\label{Fig:ex_pt_recon}
\end{figure}

\subsection{Reconstruction of line source}\label{subsec:num_curve}
In this part, we  assume that $\mu(x)=\chi_C(x)$  is a line source supported on a curve $C$, which is parameterized as follows:
\begin{equation*} 
C(\theta)=a_0+\sum_{i=1}^{N} a_i\cos(\omega_i \theta)+b_i\sin(\omega_i\theta),\quad \theta\in[0,2\pi].
\end{equation*}
Here, the frequencies $\{\omega_i\}_{i=1}^N$ are assumed to be known \textit{a priori}, and our objective is to reconstruct the unknown coefficients $\boldsymbol{a}=(a_0,\cdots,a_N)$ and $\boldsymbol{b}=(b_0,\cdots,b_N)$. Since the coefficients corresponding to different frequencies  have different influences on the measurement, we employ the Levenberg-Marquardt method \cite{Levenberg:1944,Marquardt:1963} for the numerical reconstruction. 

We define a nonlinear operator $F: (\boldsymbol{a},\boldsymbol{b}) \in \mathbb{R}^{2N+1}\to (\partial_\nu u(t,x_1),\partial_\nu u(t,x_2))\in (L^2(0,T))^2 $, where $u$ solves problem \eqref{eqmu} corresponding to coefficients  $(\boldsymbol{a},\boldsymbol{b})$. For an initial guess $(\boldsymbol{a}^0,\boldsymbol{b}^0)$, we implement the following iteration:
\begin{equation*}
    (\boldsymbol{a}^{k+1},\boldsymbol{b}^{k+1})=\arg\min J_k(\boldsymbol{a},\boldsymbol{b}), 
\end{equation*} 
with the quadratic functional 
\begin{align*}
    J_k(\boldsymbol{a},\boldsymbol{b})=&\frac{1}{2}\|F(\boldsymbol{a}^{k},\boldsymbol{b}^{k})-z^\delta + \partial_{\boldsymbol{a}} F(\boldsymbol{a}^{k},\boldsymbol{b}^{k})(\boldsymbol{a}-\boldsymbol{a}^k)  + \partial_{\boldsymbol{b}} F(\boldsymbol{a}^{k},\boldsymbol{b}^{k})(\boldsymbol{b}-\boldsymbol{b}^k)  \|_{L^2(0,T)}^2\\
    &+ \frac{\lambda_0^k}{2}|a_0-a_0^k|^2 + \sum_{i=1}^{N}\frac{\lambda_i^k}{2}\left(|a_i-a_i^k|^2+ |b_i-b_i^k|^2\right),
\end{align*}
where $\lambda_0^k,\cdots\lambda_N^k$  are regularization parameters. We employ $N+1$ parameters since $a_0,\cdots,a_N,b_1,\cdots,b_N$ influence the data  differently. The Jacobians $\partial_{\boldsymbol{a}} F $ and $\partial_{\boldsymbol{b}} F $ are approximated by the central finite difference scheme with step size $\epsilon=1\times10^{-3}$. 

Now we present numerical results for the line source identification. The
accuracy of a reconstruction is measured by the relative $\ell^2$ error: $e_\mu(\boldsymbol{a}^k,\boldsymbol{b}^k)=|(\boldsymbol{a}^k,\boldsymbol{b}^k)-(\boldsymbol{a}^\dagger,\boldsymbol{b}^\dagger)|/|(\boldsymbol{a}^\dagger,\boldsymbol{b}^\dagger)|$, where $(\boldsymbol{a}^\dagger,\boldsymbol{b}^\dagger) $ refers to the exact coefficients of curve $C$. The $\ell^2$ norm is equivalent to $L^2$ norm by Fourier series expansion.  The residual $r_\mu$  of the recovered source is computed as $r_\mu(\boldsymbol{a}^k,\boldsymbol{b}^k)=\|F(\boldsymbol{a}^k,\boldsymbol{b}^k)-z^\delta\|_{L^2(0,T)}$.

\begin{example}\label{ex:curve}
    Let $\sigma(t)= \sin(\pi t/0.8)^2  \chi_{[0,0.8]}(t)$, where $\chi$ is the characteristic function. The boundary observation points are $x_\ell=e^{i\theta_\ell}$, with $\theta_1=0$, $\theta_2=\sqrt{2}/2$. 
    \begin{enumerate} 
      \item[(i)]  $C(\theta)=0.5+0.1\cos(2\pi \theta)+0.1\sin( 2\pi \theta) $;
      \item[(ii)] $C(\theta)=0.5+0.1\cos(5\pi \theta)+0.1\sin( 5\pi \theta) $;
      \item[(iii)]  $C(\theta)=0.4+0.1\cos(2\pi \theta)+0.1\sin( 2\pi \theta) +0.1\cos(5\pi \theta)+0.1\sin( 5\pi \theta) $.
    \end{enumerate}
\end{example}

In Figure~\ref{Fig:ex_curve_recon}, we show the convergence behaviour  of the Levenberg-Marquardt method. The regularization parameters $\lambda_i^k$ are chosen  decreasing geometrically: $\lambda_0^k =0.0001\times 0.9^k $, $\lambda_1^k =0.1\times 0.9^k $  for Example~\ref{ex:curve}(i), $\lambda_0^k =0.0001\times 0.9^k$, $\lambda_1^k =0.5\times 0.9^k $ for Example~\ref{ex:curve}(ii), and $\lambda_0^k =0.0001\times 0.9^k$, $\lambda_1^k =1\times 0.9^k $, $\lambda_2^k =0.1\times 0.9^k $ for Example~\ref{ex:curve}(iii).  

\begin{table}[hbt!]
  \centering
  \begin{threeparttable}
    \caption{The relative errors for Example~\ref{ex:curve} at different noise levels.\label{tab:ex_curve}}
    \begin{tabular}{c|cccc}
        \toprule
           $e_\mu$ &  Example~\ref{ex:curve}(i) &  Example~\ref{ex:curve}(ii) &  Example~\ref{ex:curve}(iii)   \\
        \midrule
             $\delta=1\%$ & 4.01e-2 & 7.21e-2 & 6.31e-2  \\
             $\delta=5\%$ & 4.71e-2 & 8.72e-2 & 7.77e-2   \\ 
        \bottomrule
    \end{tabular}
      \end{threeparttable}
\end{table}

\begin{figure}[htbp]
	\centering
	\begin{tabular}{ccc}
		\includegraphics[width=0.25\textwidth]{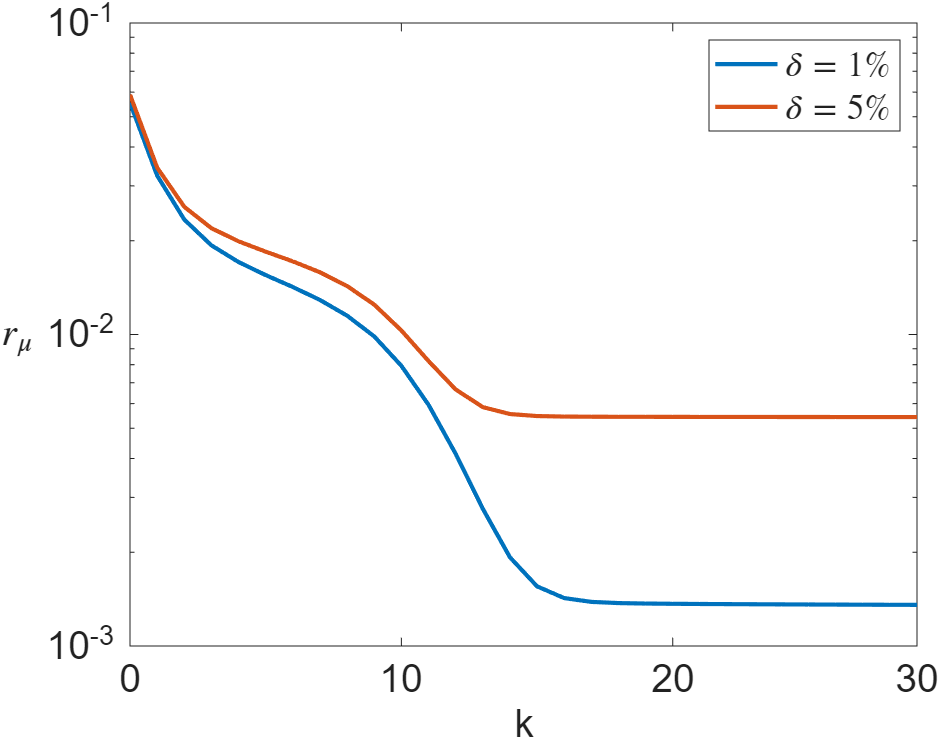}&
		\includegraphics[width=0.25\textwidth]{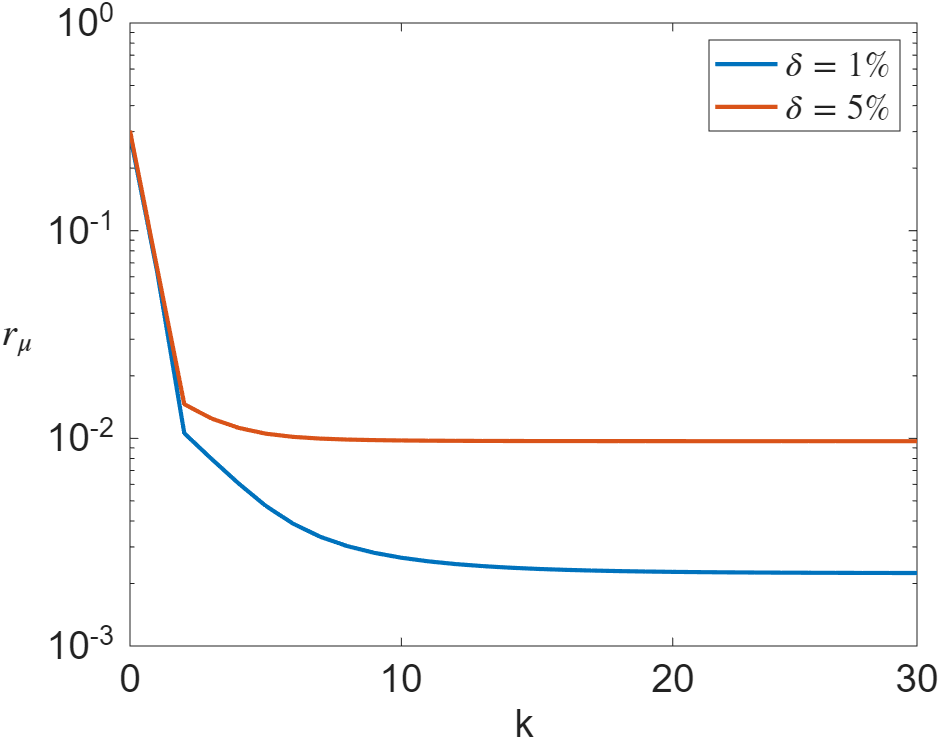}&
		\includegraphics[width=0.25\textwidth]{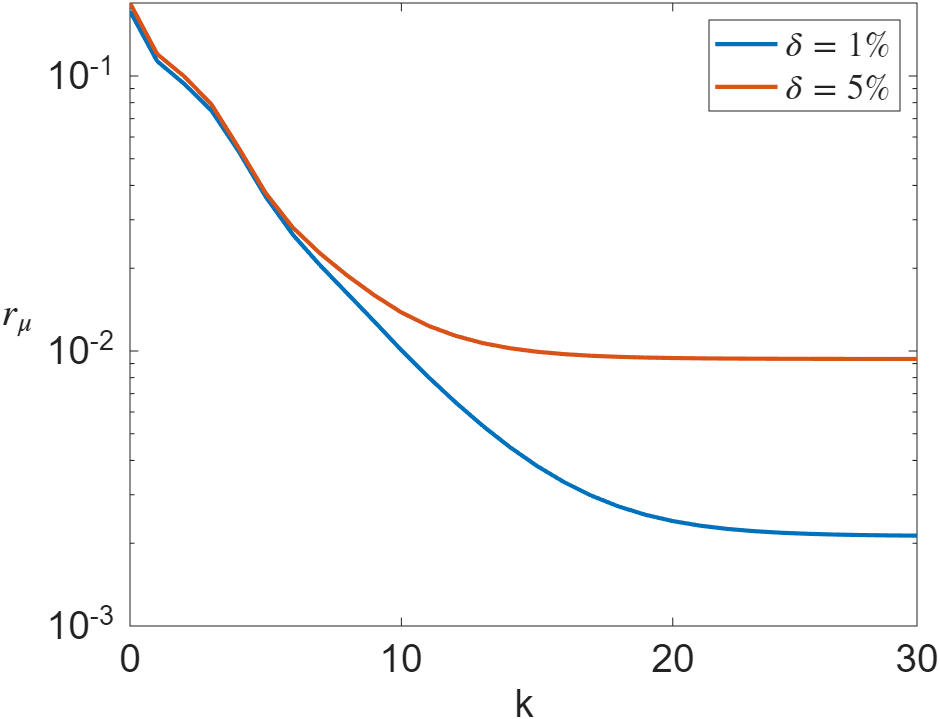} \\
        \includegraphics[width=0.25\textwidth]{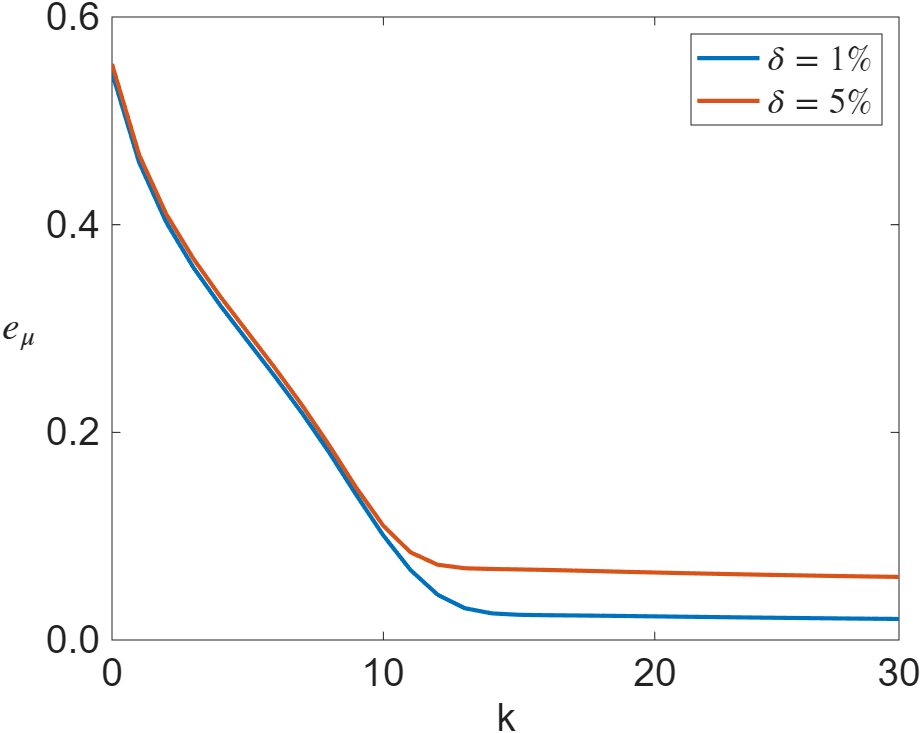}&
		\includegraphics[width=0.25\textwidth]{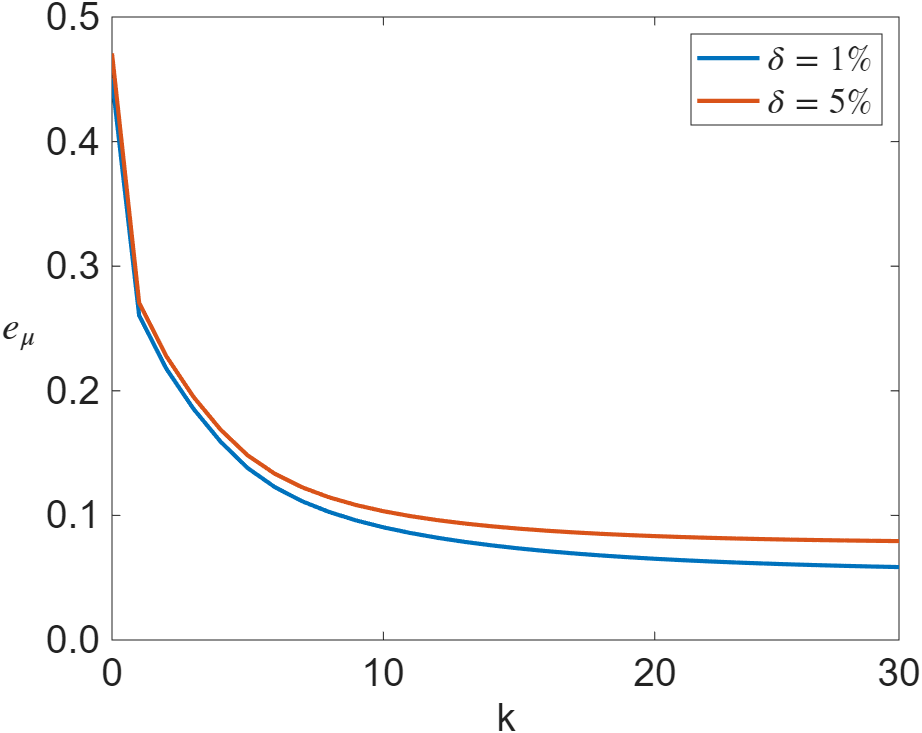}&
		\includegraphics[width=0.25\textwidth]{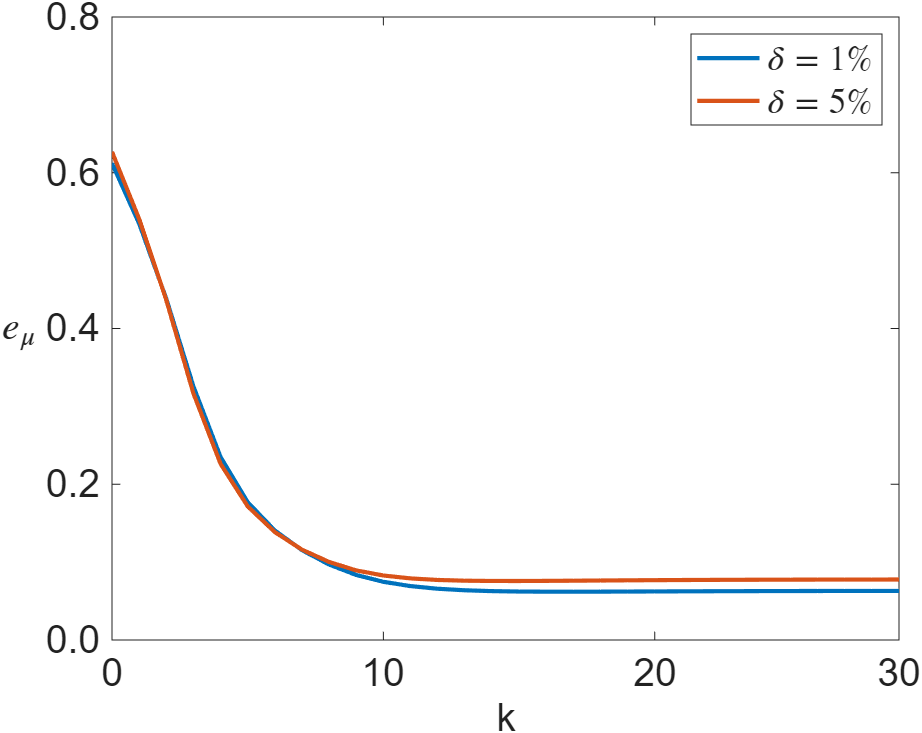} \\
		(a) Example~\ref{ex:curve}(i) & (b) Example~\ref{ex:curve}(ii)  & (c) Example~\ref{ex:curve}(iii) 
	\end{tabular}
	\caption{Evolution of residuals $r_\mu$ and relative errors $e_\mu$ with respect to iteration number $k$ in Example~\ref{ex:curve}.}
	\label{Fig:ex_curve_loss}
\end{figure}

\begin{figure}[htbp]
	\centering
	\begin{tabular}{ccc}
		\includegraphics[width=0.22\textwidth]{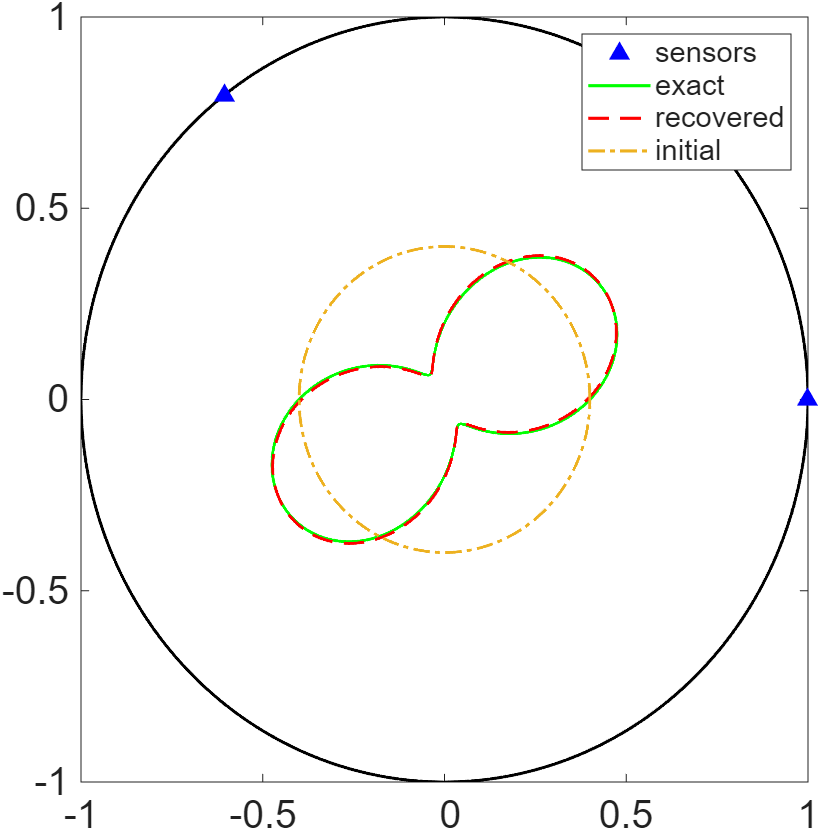}&
		\includegraphics[width=0.22\textwidth]{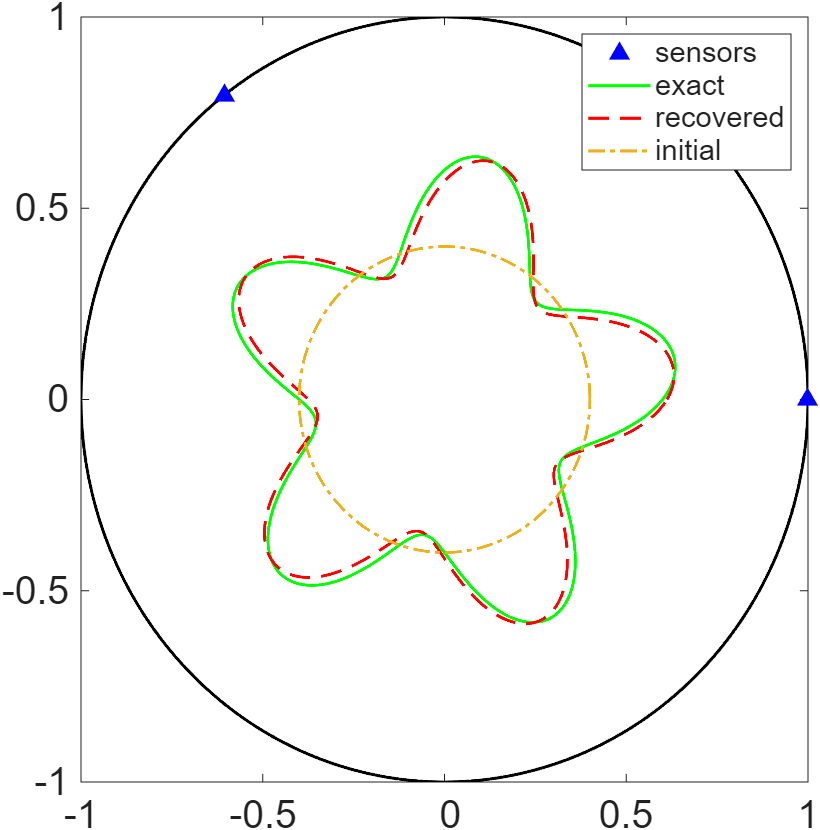}&
		\includegraphics[width=0.22\textwidth]{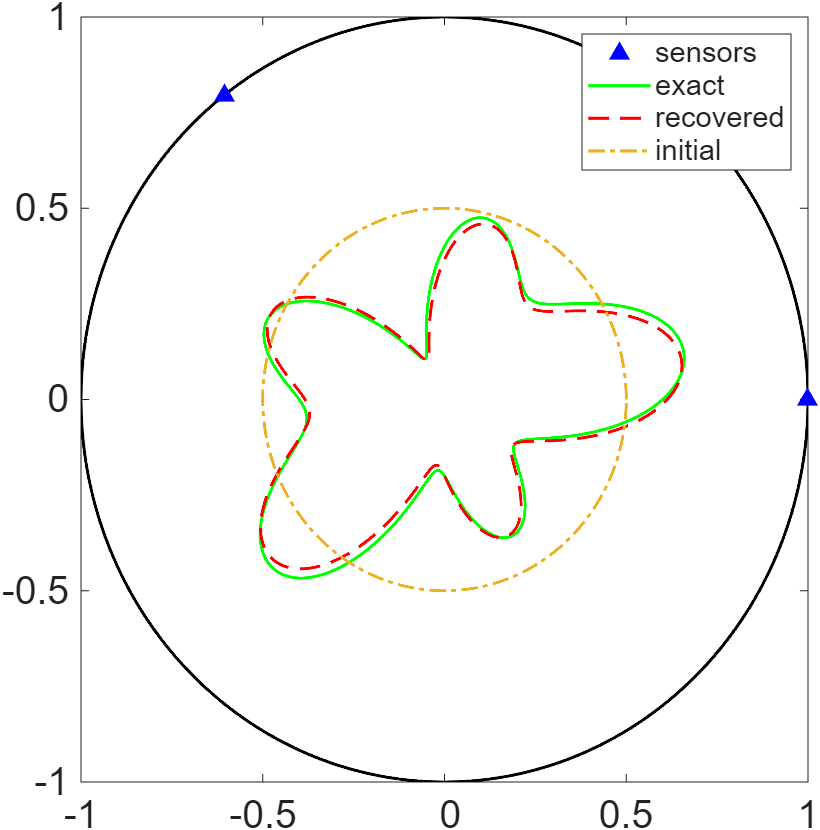} \\
		\includegraphics[width=0.22\textwidth]{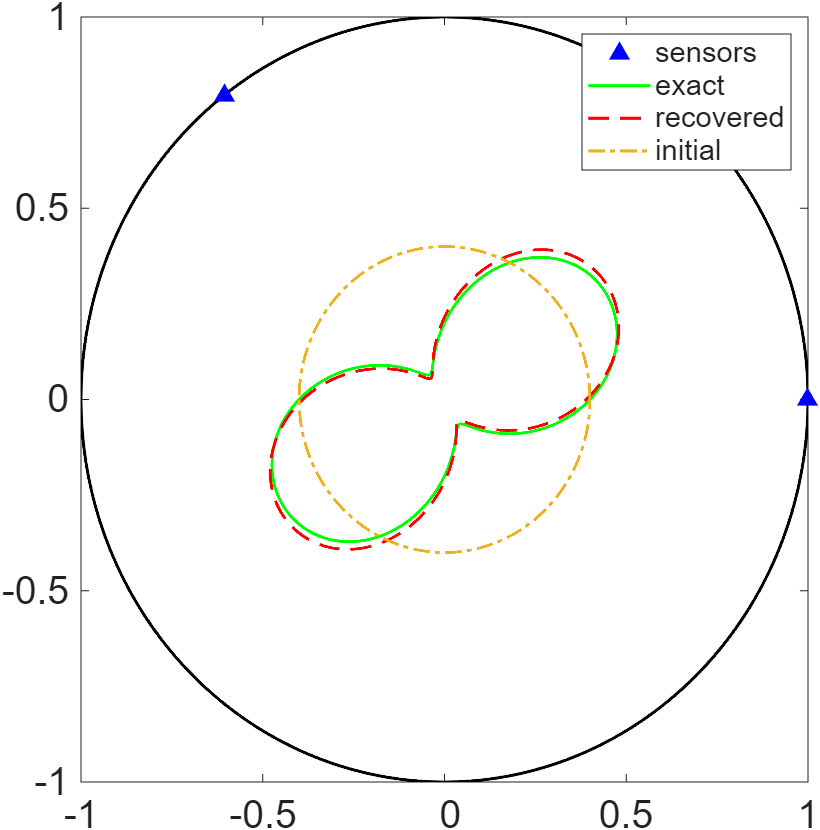}&
		\includegraphics[width=0.22\textwidth]{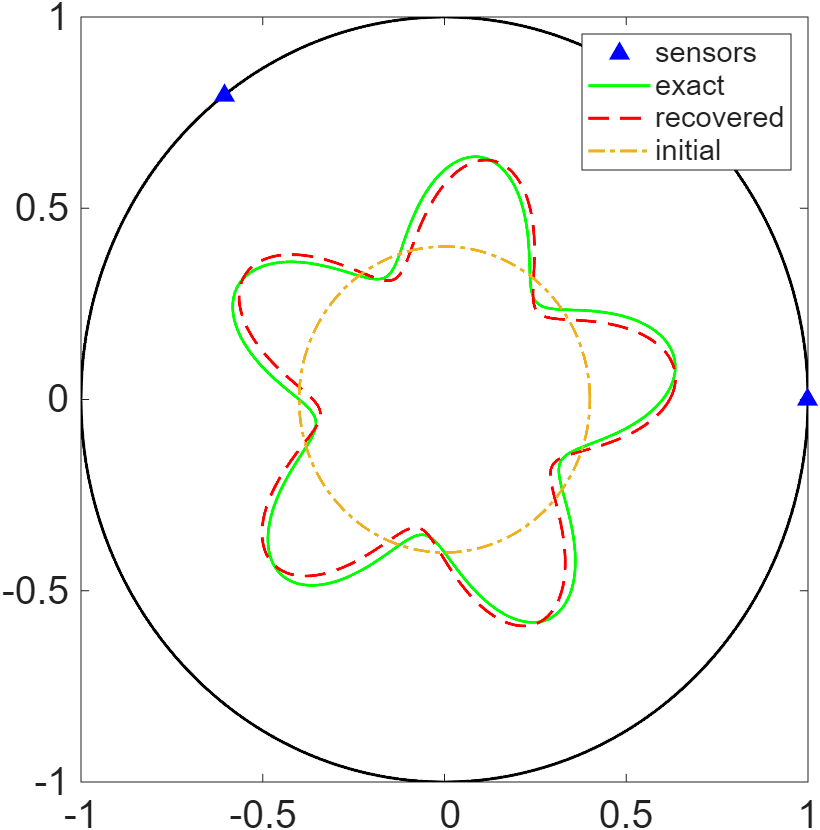}&
		\includegraphics[width=0.22\textwidth]{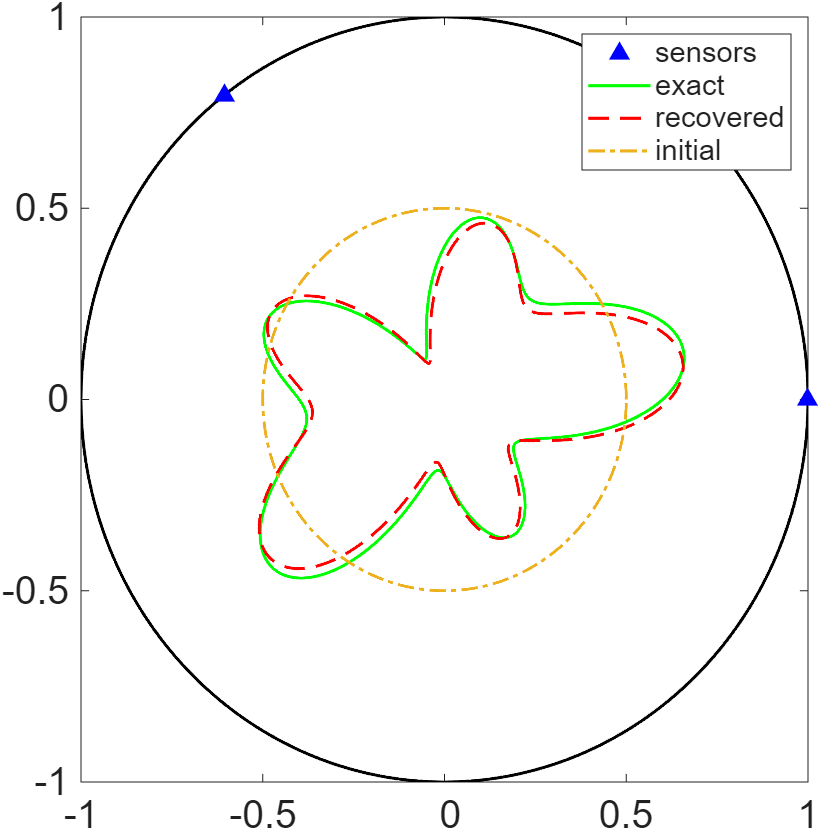} \\ 
		(a) Example~\ref{ex:curve}(i) & (b) Example~\ref{ex:curve}(ii)  & (c) Example~\ref{ex:curve}(iii) 
	\end{tabular}
	\caption{The reconstruction  for Example~\ref{ex:curve} at different noise levels. First row: $\delta=1\%$. Second row: $\delta=5\%$.}
	\label{Fig:ex_curve_recon}
\end{figure}

\subsection*{Acknowledgements}
The work of the three authors is supported by the French National Research Agency ANR and Hong Kong RGC Joint Research Scheme for the project IdiAnoDiff (grant ANR-24-CE40-7039).

\appendix
\section{Appendix}

\subsection{Properties of finite signed Radon measures}
Similarly to Section \ref{sectionproof}, for all $n\in\mathbb N$, we  consider an orthonormal basis $\{\phi_{n,k}\}_{k=1}^{d_n}$ of eigenspace of $A$ associated with the eigenvalues $\lambda_n$. We recall also that $\{\phi_{n,k}: n\in \mathbb N,\ k=1,\ldots,d_n\}$ is an orthonormal basis of $L^2(\Omega)$ of eigenvectors of the operator $A$. The main goal of this subsection, is to prove the following result about the unique determination of a  finite signed Radon measure from his spectral projections. 
\begin{lem}\label{nullmeasure}
	Set $\delta>0$ and let $\mu$ be a finite signed Radon measure with $\supp\mu\subset\subset\Omega$ satisfying :
	\begin{align}\label{null_mu_product}
	\scalarprod{\mu,\phi_{n,k}}= 0,\quad \forall n\in\N,\ k=1,\ldots,d_n.
	\end{align}
Then, $\mu$ is a null measure.
\end{lem}
\begin{proof}
    The proof is divided into two steps.\\
  \textit{Step 1 :}  We will show that $\scalarprod{\mu,\psi}_{\mathcal{D}'(\Omega),\,\mathcal{C}_0^{\infty}(\Omega)} = 0$ for every $\psi\in\mathcal{C}_0^{\infty}(\Omega)$. Let $\psi\in\mathcal{C}_0^{\infty}(\Omega)$ and notice that $\psi\in X^{1+\delta}$ with $\psi = \sum\limits_{n=1}^{\infty}\scalarprod{\psi,\phi_n}\phi_n$ which converges in $X^{2}$. Applying \eqref{mureg} and \eqref{null_mu_product}, we get
    \begin{align*}
        \scalarprod{\mu,\psi}_{\mathcal{D}'(\Omega),\,\mathcal{C}_0^{\infty}(\Omega)} = \scalarprod{\mu,\psi}_{-1,1}= \sum_{n= 1}^{+\infty}\sum_{k=1}^{d_n}\scalarprod{\psi,\phi_{n,k}}\underbrace{\scalarprod{\mu,\phi_{n,k}}_{-1,1}}_{=0}=0.
    \end{align*}
Thus, we have $\scalarprod{\mu,\psi}_{\mathcal{D}'(\Omega),\,\mathcal{C}_0^{\infty}(\Omega)} = 0$.\\

    \textit{Step 2 :}  We show that $\mu$ is a null measure. Let $\psi\in\mathcal{C}(\overline{\Omega})$, since $\supp\mu\subset\subset\Omega$ we fix $\mathcal{O}\subset\subset\Omega$ a neighbourhood of $\supp\mu$. Applying the Stone-Weierstrass theorem, one can check that $\{g_{|_{\overline{O}}}\; ;\; g\in\mathcal{C}_c^{\infty}(\Omega)\}$ is dense in $\mathcal{C}(\overline{O})$. Thus we can find a sequence $(\psi_m)_{m\in\N}$ of $\mathcal{C}_c^{\infty}(\overline{\Omega})$ such that $\sup\limits_{x\in\overline{\mathcal{O}}}\abs{\psi_m(x)-\psi(x)}\xrightarrow[n\rightarrow+\infty]{} 0$. Then, using the conclusion of Step 1, we find
      \begin{align*}
        \scalarprod{\mu,\psi}_{\mathcal{C}(\overline{\Omega})',\,\mathcal{C}(\overline{\Omega})} 
        = \int_\mathcal O \psi(x) d\mu(x)=\lim_{m\to+\infty }\int_\mathcal O \psi_m(x) d\mu(x)
        =\lim_{m\to+\infty }\scalarprod{\mu,\psi_m}_{\mathcal{D}'(\Omega),\,\mathcal{C}_c^{\infty}(\Omega)}=0.
    \end{align*}
This clearly implies that the finite signed Radon measure $\mu$ is a null measure.
    
\end{proof}

\subsection{Improvement for more regular source}
In this section, we consider problem \textbf{(ISP')} for source $\mu\in H^s(\Omega)$ with $s\in[0,1/2)$. We will see that, in this case the condition imposed to the time dependent source  $\sigma$ can be relaxed. The main result of this section can be stated as follows.

\begin{cor}\label{c1}
Set $\{x_\ell = e^{i\theta_\ell}\}_{\ell=1,2}\subset\partial\Omega$ satisfying  $\theta_1 - \theta_2\notin\pi\Q$ and, for $j=1,2$, let $\mu_j\in H^s(\Omega)$, $s\in[0,1/2)$, and $\sigma\in L^2(0,T)$ be a non-uniformly vanishing function and  constant on $(T_1,T)$, for some $T_1\in(0,T)$. Assume also that one of the following conditions is fulfilled: (i) $\mu_j\in L^2(\Omega)$, $\supp{\mu_j}\subset\subset\Omega$, $j=1,2$ and $\sigma\in H^1(0,T)$; (ii) $\mu_1,\mu_2\in H^s(\Omega)$, $s\in(0,1/2)$, and $\sigma\in L^2(0,T)$. Then condition \eqref{Hypothesis_normalderivative} implies that $\mu_1=\mu_2$.    
\end{cor}
\begin{proof}
We assume that \eqref{Hypothesis_normalderivative} is fulfilled and, under assumption (i) or (ii), we will prove that $\mu_1=\mu_2$. 
 We consider first that condition (i) is fulfilled. Since $\mu=\mu_1-\mu_2\in L^2(\Omega)$, fixing $u=u_1-u_2$ and repeating
the argumentation of Proposition \ref{bound}, we can derive a spectral representation for the term $\partial_{\nu}u$ that takes the following form
  $$
        \partial_{\nu}u(t,z) = \sigma(t)\partial_{\nu}G(z) - \sum_{n=1}^{+\infty}\left(\sigma(0)\frac{\scalarprod{\mu,\phi_n}}{\lambda_n}e^{-t\lambda_n}+\int_0^te^{-\lambda_n(t-s)}\sigma'(s)\frac{\scalarprod{\mu,\phi_n}}{\lambda_n}ds\right)\partial_{\nu}\phi_n(z),\quad (t,z)\in \Sigma.$$
Combining this formula with the fact that, for all $z\in\partial\Omega$, the map $t\mapsto \sum_{n=1}^{+\infty}\sigma(0)\frac{\scalarprod{\mu,\phi_n}}{\lambda_n}e^{-t\lambda_n}\partial_{\nu}\phi_n(z)$ is analytic on $(0,+\infty)$ and repeating the arguments of Theorem \ref{inv_prob}, one can prove that $\mu_1=\mu_2$.

Now let us consider the situation where condition (ii) is fulfilled.
In light of \cite[Chapter~1,~Theorem~11.4]{lionsmagenesvol1} and \cite[Chapter~1,~Theorem~11.6]{lionsmagenesvol1}, we have $\mu=\mu_1-\mu_2\in H^s(\Omega)=X^s$. Then, following the argumentation of   Proposition \ref{bound}, we can
prove that $u\in L^2(0,T;X^{2+s})$ and we can derive a spectral representation for the term $\partial_{\nu}u$ that takes the following form
$$
       \begin{aligned} \partial_{\nu}u(t,z) &=  \sum_{n=1}^{+\infty}\left(\int_0^te^{-\lambda_n(t-s)}\sigma(s)\scalarprod{\mu,\phi_n}ds\right)\partial_{\nu}\phi_n(z)\\
       &=\sum_{n=1}^{+\infty}\left(\sigma(T)\frac{\scalarprod{\mu,\phi_n}}{\lambda_n}(1-e^{-t\lambda_n})+\int_0^te^{-\lambda_n(t-s)}(\sigma(s)-\sigma(T))\scalarprod{\mu,\phi_n}ds\right)\partial_{\nu}\phi_n(z),\quad (t,z)\in \Sigma.\end{aligned}$$
Combining this formula with the fact that $\sigma-\sigma(T)=0$ on $(T_1,T)$ and applying the arguments used in the proof of Theorem \ref{inv_prob}, one can check that $\mu_1=\mu_2$.

\end{proof}

\bibliographystyle{plain}
\bibliography{biblio_diffusion}

@book{watson1922treatise,
  title={A treatise on the theory of Bessel functions},
  author={Watson, George Neville},
  volume={3},
  year={1922},
  publisher={The University Press}
}

@book{lionsmagenesvol1,
 author = {Lions, J. L. and Magenes, E.},
 title = {Non-homogeneous boundary value problems and applications. {Vol}. {I}. {Translated} from the {French} by {P}. {Kenneth}},
 fseries = {Grundlehren der Mathematischen Wissenschaften},
 series = {Grundlehren Math. Wiss.},
 issn = {0072-7830},
 volume = {181},
 year = {1972},
 publisher = {Springer, Cham},
 language = {English},
 zbMATH = {3353865},
 Zbl = {0223.35039}
}

@book{rud,
 author = {Rudin, Walter},
 title = {Real and complex analysis.},
 edition = {3rd ed.},
 isbn = {0-07-054234-1},
 year = {1987},
 publisher = {New York, NY: McGraw-Hill},
 language = {English},
 zbMATH = {1022658},
 Zbl = {0925.00005}
}

@book{lionsmagenesvol2,
 author = {Lions, J. L. and Magenes, E.},
 title = {Non-homogeneous boundary value problems and applications. {Vol}. {II}. {Translated} from the {French} by {P}. {Kenneth}},
 fseries = {Grundlehren der Mathematischen Wissenschaften},
 series = {Grundlehren Math. Wiss.},
 issn = {0072-7830},
 volume = {182},
 year = {1972},
 publisher = {Springer, Cham},
 language = {English},
 zbMATH = {3360568},
 Zbl = {0227.35001}
}

@article{JET,
 author = {Al Jebawy, Hanin and Elbadia, Abdellatif and Triki, Faouzi},
 title = {Inverse moving point source problem for the wave equation},
 fjournal = {Inverse Problems},
 journal = {Inverse Problems},
 issn = {0266-5611},
 volume = {38},
 number = {12},
 pages = {125003, 11 pp.}, 
 year = {2022},
 language = {English},
 doi = {10.1088/1361-6420/ac9999},
 zbMATH = {7611869},
 Zbl = {1501.35452}
}

@article{WKT,
 author = {Wang, S{\'a}ra and Karamehmedovi{\'c}, Mirza and Triki, Faouzi},
 title = {Localization of moving sources: uniqueness, stability, and {Bayesian} inference},
 fjournal = {SIAM Journal on Applied Mathematics},
 journal = {SIAM J. Appl. Math.},
 issn = {0036-1399},
 volume = {83},
 number = {3},
 pages = {1049--1073},
 year = {2023},
 language = {English},
 doi = {10.1137/22M1500691},
 zbMATH = {7699954},
 Zbl = {1518.35691}
}

@article{MGH,
 author = {Ma, Guanqiu and Guo, Hongxia and Hu, Guanghui},
 title = {Imaging a moving point source from multifrequency data measured at one and sparse observation points. {II}: {Near}-field case in {3D}},
 fjournal = {SIAM Journal on Imaging Sciences},
 journal = {SIAM J. Imaging Sci.},
 issn = {1936-4954},
 volume = {17},
 number = {3},
 pages = {1377--1414},
 year = {2024},
 language = {English},
 doi = {10.1137/23M162260X},
 zbMATH = {7896977},
 Zbl = {1547.35202}
}

@article{gong2026identification,
  TITLE = {{Identification of a Point Source in the Heat Equation from Sparse Boundary Measurements }},
  AUTHOR = {Gong, Fangyu and Jin, Bangti and Kian, Yavar and Liu, Sizhe},
  URL = {https://hal.science/hal-05545664},
   JOURNAL = {SIAM J. Math. Anal.},
  FJOURNAL = {SIAM Journal on Mathematical Analysis},
  PUBLISHER = {{Society for Industrial and Applied Mathematics}},
  YEAR = {2026 in press},
  HAL_ID = {hal-05545664},
  HAL_VERSION = {v1},
}

@book {Thomee:2006,
    AUTHOR = {Thom\'{e}e, Vidar},
     TITLE = {Galerkin {F}inite {E}lement {M}ethods for {P}arabolic {P}roblems},
   EDITION = {Second},
 PUBLISHER = {Springer-Verlag, Berlin},
      YEAR = {2006},
     PAGES = {xii+370},
      ISBN = {978-3-540-33121-6; 3-540-33121-2},
   MRCLASS = {65-02 (65M15 65M60)},
  MRNUMBER = {2249024},
}

@book{Is1,
 author = {Isakov, Victor},
 title = {Inverse source problems},
 fseries = {Mathematical Surveys and Monographs},
 series = {Math. Surv. Monogr.},
 issn = {0076-5376},
 volume = {34},
 isbn = {0-8218-1532-6},
 year = {1990},
 publisher = {Providence, RI: American Mathematical Society},
 language = {English},
 zbMATH = {194472},
 Zbl = {0721.31002}
}

@book{Nocedal:2006,
 author = {Nocedal, Jorge and Wright, Stephen J.},
 title = {Numerical optimization},
 edition = {2nd ed.},
 fseries = {Springer Series in Operations Research and Financial Engineering},
 series = {Springer Ser. Oper. Res. Financ. Eng.},
 issn = {1431-8598},
 isbn = {0-387-30303-0},
 year = {2006},
 publisher = {New York, NY: Springer},
 language = {English},
 zbMATH = {5060482},
 Zbl = {1104.65059}
}

@article {Marquardt:1963,
    AUTHOR = {Marquardt, Donald W.},
     TITLE = {An algorithm for least-squares estimation of nonlinear
              parameters},
   JOURNAL = {J. Soc. Indust. Appl. Math.},
  FJOURNAL = {Journal of the Society for Industrial and Applied Mathematics},
    VOLUME = {11},
      YEAR = {1963},
     PAGES = {431--441},
      ISSN = {0368-4245},
   MRCLASS = {62.20},
  MRNUMBER = {153071},
MRREVIEWER = {M.\ Atiqullah},
}

@article {Levenberg:1944,
    AUTHOR = {Levenberg, Kenneth},
     TITLE = {A method for the solution of certain non-linear problems in
              least squares},
   JOURNAL = {Quart. Appl. Math.},
  FJOURNAL = {Quarterly of Applied Mathematics},
    VOLUME = {2},
      YEAR = {1944},
     PAGES = {164--168},
      ISSN = {0033-569X,1552-4485},
   MRCLASS = {65.0X},
  MRNUMBER = {10666},
MRREVIEWER = {T.\ E.\ Sterne},
       DOI = {10.1090/qam/10666},
       URL = {https://doi.org/10.1090/qam/10666},
}

@article{HR,
 author = {Hettlich, F. and Rundell, W.},
 title = {Identification of a discontinuous source in the heat equation},
 fjournal = {Inverse Problems},
 journal = {Inverse Problems},
 issn = {0266-5611},
 volume = {17},
 number = {5},
 pages = {1465--1482},
 year = {2001},
 language = {English},
 doi = {10.1088/0266-5611/17/5/315},
 zbMATH = {1691018},
 Zbl = {0986.35129}
}

@article{GZZ,
 author = {Gu, Qiling and Zhang, Wenlong and Zhang, Zhidong},
 title = {Determine the point source of the heat equation with sparse boundary measurements},
 fjournal = {SIAM Journal on Applied Mathematics},
 journal = {SIAM J. Appl. Math.},
 issn = {0036-1399},
 volume = {85},
 number = {5},
 pages = {2337--2354},
 year = {2025},
 language = {English},
 doi = {10.1137/25M1725620},
 zbMATH = {8113966},
 Zbl = {1577.35435}
}

@article{Run2020Heatequation_source_problem,
author = {Rundell, William and Zhang, Zhidong},
title = {On the Identification of Source Term in the Heat Equation from Sparse Data},
journal = {SIAM J.  Math. Anal.},
volume = {52},
number = {2},
pages = {1526-1548},
year = {2020},
doi = {10.1137/19M1279915},

URL = { 
    
        https://doi.org/10.1137/19M1279915
},
eprint = { 
    
        https://doi.org/10.1137/19M1279915
},
}

@article {ElBadiaHaDuong2002pollution_detection_problem,
    AUTHOR = {El Badia, A. and Ha-Duong, T.},
     TITLE = {On an inverse source problem for the heat equation.
              {A}pplication to a pollution detection problem},
   JOURNAL = {J. Inverse Ill-Posed Probl.},
  FJOURNAL = {Journal of Inverse and Ill-Posed Problems},
    VOLUME = {10},
      YEAR = {2002},
    NUMBER = {6},
     PAGES = {585--599},
      ISSN = {0928-0219,1569-3945},
   MRCLASS = {35R30 (35K05 35K20 80A23)},
  MRNUMBER = {1967440},
MRREVIEWER = {Peter\ G.\ Danilaev},
       DOI = {10.1515/jiip.2002.10.6.585},
       URL = {https://doi-org.ezproxy.lb.polyu.edu.hk/10.1515/jiip.2002.10.6.585},
}

@article {Andrle2012moving_pollution_sources,
    AUTHOR = {Andrle, M. and El Badia, A.},
     TITLE = {Identification of multiple moving pollution sources in surface
              waters or atmospheric media with boundary observations},
   JOURNAL = {Inverse Problems},
  FJOURNAL = {Inverse Problems. An International Journal on the Theory and
              Practice of Inverse Problems, Inverse Methods and Computerized
              Inversion of Data},
    VOLUME = {28},
      YEAR = {2012},
    NUMBER = {7},
     PAGES = {075009, 22 pp.},
      ISSN = {0266-5611,1361-6420},
   MRCLASS = {86A10 (35K20 35R30 86A05)},
  MRNUMBER = {2946797},
MRREVIEWER = {Sven\ Ivansson},
       DOI = {10.1088/0266-5611/28/7/075009},
       URL = {https://doi-org.ezproxy.lb.polyu.edu.hk/10.1088/0266-5611/28/7/075009},
}

@article {Li2020fractionalorder_sourceterm,
    AUTHOR = {Li, Zhiyuan and Zhang, Zhidong},
     TITLE = {Unique determination of fractional order and source term in a
              fractional diffusion equation from sparse boundary data},
   JOURNAL = {Inverse Problems},
  FJOURNAL = {Inverse Problems. An International Journal on the Theory and
              Practice of Inverse Problems, Inverse Methods and Computerized
              Inversion of Data},
    VOLUME = {36},
      YEAR = {2020},
    NUMBER = {11},
     PAGES = {115013, 20 pp.},
      ISSN = {0266-5611,1361-6420},
   MRCLASS = {65M32 (35R11 47N40)},
  MRNUMBER = {4173588},
MRREVIEWER = {Francesco\ Zirilli},
       DOI = {10.1088/1361-6420/abbc5d},
       URL = {https://doi-org.ezproxy.lb.polyu.edu.hk/10.1088/1361-6420/abbc5d},
}

@misc{HJKT,
 author = {Kuang Huang and Bangti Jin and Yavar Kian and Faouzi Triki},
 title = {Stability Estimates for the Inverse Problem of Reconstructing Point sources in Parabolic Equations},
 year = {2026},
 howpublished = {Preprint, arXiv:2603.09099},
 url = {https://arxiv.org/abs/2603.09099},
 arXiv = {arXiv:2603.09099}
}

@article{SuWa,
 author = {Sun, Shiwei and Wang, Haibing},
 title = {A novel approach for locating point sources in a diffusive medium},
 fjournal = {SIAM Journal on Applied Mathematics},
 journal = {SIAM J. Appl. Math.},
 issn = {0036-1399},
 volume = {85},
 number = {4},
 pages = {1667--1689},
 year = {2025},
 language = {English},
 doi = {10.1137/24M1692678},
 zbMATH = {8083411},
 Zbl = {1571.35464}
}

@article{Briant2011improved_line_source_air_pollution,
title = {An improved line source model for air pollutant dispersion from roadway traffic},
journal = {Atmospheric Environment},
volume = {45},
number = {24},
pages = {4099-4107},
year = {2011},
issn = {1352-2310},
doi = {https://doi.org/10.1016/j.atmosenv.2010.11.016},
url = {https://www.sciencedirect.com/science/article/pii/S1352231010009714},
author = {Régis Briant and Irène Korsakissok and Christian Seigneur},
}

@Article{Nanni2022puff_lagrangianmodel_coastal_site,
AUTHOR = {Nanni, Alessandro and Tinarelli, Gianni and Solisio, Carlo and Pozzi, Cristina},
TITLE = {Comparison between Puff and Lagrangian Particle Dispersion Models at a Complex and Coastal Site},
JOURNAL = {Atmosphere},
VOLUME = {13},
YEAR = {2022},
NUMBER = {4},
ARTICLE-NUMBER = {508},
URL = {https://www.mdpi.com/2073-4433/13/4/508},
ISSN = {2073-4433},
DOI = {10.3390/atmos13040508},
}

@article{Chatterjee2020groundwater_cities_finite_sources,
  title={Groundwater contamination in mega cities with finite sources},
  author={Chatterjee, Ayan and Singh, Mritunjay Kumar and Singh, Vijay P},
  journal={Journal of Earth System Science},
  volume={129},
  number={1},
  pages={1},
  year={2020},
  publisher={Springer}
}

@article{Zheng2024line_contaminant_sources_groundwater,
title = {Estimating line contaminant sources in non-Gaussian groundwater conductivity fields using deep learning-based framework},
journal = {Journal of Hydrology},
volume = {630},
pages = {130727},
year = {2024},
issn = {0022-1694},
doi = {https://doi.org/10.1016/j.jhydrol.2024.130727},
url = {https://www.sciencedirect.com/science/article/pii/S0022169424001215},
author = {Na Zheng and Zhi Li and Xuemin Xia and Simin Gu and Xianwen Li and Simin Jiang},
}

@article{Grebenkov2013laplacian_eigenfunctions,
author = {Grebenkov, D. S. and Nguyen, B.-T.},
title = {Geometrical Structure of Laplacian Eigenfunctions},
journal = {SIAM Review},
volume = {55},
number = {4},
pages = {601-667},
year = {2013},
doi = {10.1137/120880173},

URL = { 
    
        https://doi.org/10.1137/120880173
},
eprint = { 
    
        https://doi.org/10.1137/120880173
},
}

@article{IY,
 author = {Imanuvilov, Oleg Yu and Yamamoto, Masahiro},
 title = {Lipschitz stability in inverse parabolic problems by the {Carleman} estimate},
 fjournal = {Inverse Problems},
 journal = {Inverse Problems},
 issn = {0266-5611},
 volume = {14},
 number = {5},
 pages = {1229--1245},
 year = {1998},
 language = {English},
 doi = {10.1088/0266-5611/14/5/009},
 zbMATH = {1223316},
 Zbl = {0992.35110}
}

@misc{CJKZ,
 author = {Siyu Cen and Bangti Jin and Yavar Kian and Zhi Zhou},
 title = {Numerical Analysis of Space-Time Dependent Source Identification in Subdiffusion Equations},
 year = {2026},
 howpublished = {Preprint, arXiv:2605.05579 },
 url = {https://arxiv.org/abs/2605.05579},
 arXiv = {arXiv:2605.05579}
}

@article{JKZ1,
 author = {Jin, Bangti and Kian, Yavar and Zhou, Zhi},
 title = {Reconstruction of a space-time-dependent source in subdiffusion models via a perturbation approach},
 fjournal = {SIAM Journal on Mathematical Analysis},
 journal = {SIAM J. Math. Anal.},
 issn = {0036-1410},
 volume = {53},
 number = {4},
 pages = {4445--4473},
 year = {2021},
 language = {English},
 doi = {10.1137/21M1397295},
 url = {hdl.handle.net/10397/93852},
 zbMATH = {7381135},
 Zbl = {1470.35433}
}

@article{JKZ,
 author = {Jin, Bangti and Kian, Yavar and Zhou, Zhi},
 title = {Inverse problems for subdiffusion from observation at an unknown terminal time},
 fjournal = {SIAM Journal on Applied Mathematics},
 journal = {SIAM J. Appl. Math.},
 issn = {0036-1399},
 volume = {83},
 number = {4},
 pages = {1496--1517},
 year = {2023},
 language = {English},
 doi = {10.1137/22M1529105},
 url = {discovery.ucl.ac.uk/id/eprint/10175946/},
 zbMATH = {7723859},
 Zbl = {1525.35247}
}

@article{KLY,
 author = {Kian, Yavar and Liu, Yikan and Yamamoto, Masahiro},
 title = {Uniqueness of inverse source problems for general evolution equations},
 fjournal = {Communications in Contemporary Mathematics},
 journal = {Commun. Contemp. Math.},
 issn = {0219-1997},
 volume = {25},
 number = {6},
 pages = {2250009, 33 pp.}, 
 year = {2023},
 language = {English},
 doi = {10.1142/S0219199722500092},
 zbMATH = {7696881},
 Zbl = {1517.35244}
}

@article{KSXY,
 author = {Kian, Yavar and Soccorsi, {\'E}ric and Xue, Qi and Yamamoto, Masahiro},
 title = {Identification of time-varying source term in time-fractional diffusion equations},
 fjournal = {Communications in Mathematical Sciences},
 journal = {Commun. Math. Sci.},
 issn = {1539-6746},
 volume = {20},
 number = {1},
 pages = {53--84},
 year = {2022},
 language = {English},
 doi = {10.4310/CMS.2022.v20.n1.a2},
 zbMATH = {7474593},
 Zbl = {1483.35342}
}

@article{Choulli2006stability_estimates,
 author = {Choulli, M. and Yamamoto, M.},
 title = {Some stability estimates in determining sources and coefficients},
 fjournal = {Journal of Inverse and Ill-Posed Problems},
 journal = {J. Inverse Ill-Posed Probl.},
 issn = {0928-0219},
 volume = {14},
 number = {4},
 pages = {355--373},
 year = {2006},
 language = {English},
 doi = {10.1515/156939406777570996},
 zbMATH = {5122057},
 Zbl = {1110.35100}}
\nocite{*} 

%

\end{document}